\documentclass[a4paper, english]{amsart}
\usepackage[utf8]{inputenc}
\usepackage{babel}
\usepackage{amsmath, amsthm, amssymb, amsfonts}
\usepackage{bm}
\usepackage{mathrsfs}
\usepackage{mathtools}
\usepackage[dvipsnames]{xcolor}
\usepackage{bbm}
\usepackage{graphicx}
\usepackage{caption}
\usepackage{subcaption}
\usepackage{todonotes}
\usepackage{nicefrac}
\usepackage{hyperref}
\usepackage{listings}
\usepackage{enumitem}
\usepackage{etoolbox}
\usepackage{chngcntr}
\usepackage{xspace}
\usepackage{microtype}
\usepackage{csquotes}
\usepackage[normalem]{ulem}

\usepackage[backend=biber,giveninits=true, eprint=true]{biblatex}
\AtEveryBibitem{
    \clearfield{urlyear}
    \clearfield{urlmonth}
    \clearfield{url}
    \clearfield{doi}
    \clearfield{isbn}
    \clearfield{issn}
    \clearfield{month}
}

\hypersetup{hidelinks}

\theoremstyle{plain}
\newtheorem{theorem}{Theorem}[section]

\newtheorem{proposition}[theorem]{Proposition}
\newtheorem{lemma}[theorem]{Lemma}

\counterwithin*{claim}{theorem}

\theoremstyle{definition}
\newtheorem{definition}[theorem]{Definition}
\newtheorem{example}[theorem]{Example}

\theoremstyle{remark}
\newtheorem{remark}[theorem]{Remark}

\numberwithin{equation}{section}

\newcommand{\R}{\mathbb{R}}

\newcommand{\N}{\mathbb{N}}
\newcommand{\C}{\mathcal{C}}
\newcommand{\E}{\mathbb{E}}

\newcommand{\F}{\mathscr{F}}

\renewcommand{\P}{\mathbb{P}}

\newcommand{\norm}[1]{\mathopen\|#1\mathclose\|}
\newcommand{\bignorm}[1]{\bigl\|#1\bigr\|}

\newcommand{\loc}{\textit{loc}}

\newcommand{\initial}{\textit{in}}
\newcommand{\Pas}{\P\text{-a.s}}
\newcommand{\Lip}{\textit{Lip}}
\newcommand{\eps}{\varepsilon}

\renewcommand{\leq}{\leqslant}
\renewcommand{\geq}{\geqslant}
\newcommand{\BV}{\textit{BV}}
\newcommand{\ind}{{\mathbbm{1}}}

\newcommand{\semimartingale}{semi\-martin\-gale\xspace}

\newcommand{\diffeomorphism}{diffeo\-morph\-ism\xspace}
\newcommand{\diffeomorphisms}{diffeo\-morph\-isms\xspace}

\DeclareMathOperator{\conv}{conv}
\DeclareMathOperator*{\esssup}{ess\, sup}

\DeclareMathOperator{\dist}{dist}

\title[Stochastic perturbation for conservation laws]{Stochastic perturbation and zero noise limit for scalar conservation laws}

\author{U. S. Fjordholm \and M. C. Ørke}

\begin{document}

\begin{abstract}
Scalar conservation laws sit at the intersection between being simple enough to study analytically, while being complex enough to exhibit a wide range of nonlinear phenomena.
We introduce a novel stochastic perturbation of scalar conservation laws, inspired by mean field games. We prove well-posedness of the stochastically perturbed equation; prove that it converges as the noise parameter is sent to 0; and that the limit is the unique entropy solution of the conservation law. Thus, the noise acts as a selection criterion for (deterministic) conservation laws. To the best of our knowledge, this is the first such result for nonlinear hyperbolic conservation laws.
\end{abstract}

\maketitle

\section{Introduction}

\subsection{Stochastic versus deterministic equations}

It has been known for several decades that, as a general rule, stochastic differential equations (SDEs) exhibit better well-posedness properties---existence, uniqueness and/or regularity under weaker assumptions---than their deterministic counterparts. This phenomenon is commonly known as \emph{regularization by noise}. For ordinary differential equations, this can be seen in the works by Zvonkin~\cite{zvonkin_1974} and Veretennikov~\cite{veretennikov_1981}, where existence and uniqueness of strong solutions of SDEs with merely $L^\infty$ velocity field (or ``drift'') is shown (see also Kunita~\cite{kunita,kunita_1990}, the stronger ``path-by-path uniqueness'' result by Davie~\cite{Dav07}, and the review paper by Gess~\cite{gess_2018}). For partial differential equations, the well-posedness results for transport equations with irregular velocity fields due to Flandoli, Gubinelli and Priola~\cite{flandoli_gubinelli_priola} stand out; see also the monograph by Flandoli~\cite{flandoli_2011} and the references therein.

In this paper we consider the scalar conservation law
\begin{equation}\label{eq:cl}
\begin{cases}
    \partial_t u + \nabla \cdot f(u) = 0 & \text{in } \R^d\times(0,T), \\
    u\bigr|_{t=0} = u_\initial & \text{on } \R^d.
\end{cases}
\end{equation}
As is well known~\cite{holden_risebro_2015,dafermos_2016,coclite_scalar_2024}, solutions of~\eqref{eq:cl} exhibit discontinuities, and the equation must therefore be read in the weak sense. Moreover, weak solutions of~\eqref{eq:cl} are generically non-unique, and various selection principles (so-called \textit{entropy conditions}) can be imposed to single out a unique, ``physically relevant'' solution (the \textit{entropy solution}).

Various stochastic perturbations of the nonlinear partial differential equation (PDE)~\eqref{eq:cl} have been studied over the years, including Holden and Risebro~\cite{holden_conservation_1997}, Feng and Nualart~\cite{feng_stochastic_2008}, and Gess and Maurelli~\cite{gess_well-posedness_2018}. Common to all of these approaches is that the entropy condition is imposed on the stochastic equation in one way or another, either by using the deterministic solution operator to construct a solution~\cite{holden_conservation_1997}, by imposing entropy conditions on the stochastic solution~\cite{feng_stochastic_2008}, or by using a kinetic formulation, where the entropy condition is implicitly baked into the equation~\cite{gess_well-posedness_2018}. To see that stochastic perturbations of~\eqref{eq:cl} are not automatically well-posed without further conditions, we recall an example due to Flandoli~\cite[Section~5.1]{flandoli_2011} and consider the stochastically perturbed equation
\begin{equation}\label{eq:cl-bad-perturbation}
    \begin{cases}
    dv^\eps + \nabla \cdot f(v^\eps)\,dt + \eps \nabla v^\eps\circ dW_t = 0 & \text{in } \R^d\times(0,T), \\
    v^\eps\bigr|_{t=0} = u_\initial & \text{on } \R^d
\end{cases}
\end{equation}
(where we have added so-called \emph{transport noise}; see Section~\ref{sec:overview}). The change of variables  $\tilde{v}^\eps_t(x)\coloneqq v_t^\eps(x+\eps W_t)$ and an application of the Ito--Wentzell--Kunita formula shows that $\tilde{v}^\eps$ is a weak solution of~\eqref{eq:cl}. (Here and elsewhere, the subscript in $v_t$ means $v$ evaluated at time $t$.) Thus, any of the infinitely many weak solutions $\tilde{v}^\eps$ of~\eqref{eq:cl} gives rise to a solution $v^\eps$ of~\eqref{eq:cl-bad-perturbation}, and the problem is therefore ill-posed \emph{per se}.

\subsection{Overview and main results}\label{sec:overview}

The purpose of the present work is twofold. First, to prove well-posedness of a novel stochastic perturbation of~\eqref{eq:cl} without explicitly imposing a selection principle. Second, to pass to the \emph{zero-noise limit} and prove that the stochastic solutions converge to the correct entropy solution of~\eqref{eq:cl}.

We will consider the stochastically perturbed equation
\begin{equation} \label{eq:mean_field_cl}
    \begin{cases}
    du^\eps + \nabla \cdot \bigl(a(m^\eps) u^\eps\bigr)\, dt + \eps \nabla u^\eps \circ dW_t = 0 &\text{in}\ \R^d \times (0, T),\\
    m^\eps \coloneqq \E[u^\eps], \\
    u^\eps\bigr|_{t = 0} = u_\initial & \text{on } \R^d,
    \end{cases}
\end{equation}
where $a(v) \coloneqq \frac{f(v)-f(0)}{v}$ for non-zero $v\in\R$ and $a(0)\coloneqq f'(0)$, $u_\initial\colon \R^d \to \R$ is the (deterministic) initial data, $(W_t)_{t \geq 0}$ is a $d$-dimensional Brownian motion, $\E[u^\eps]$ denotes the expectation of $u^\eps$, and $\eps$ is a positive parameter.
We use the notation $\nabla u \circ dW = \sum_{i = 1}^d \partial_i u \circ dW^i_t$, where $\partial_i u \circ dW^i_t$ denotes Stratonovich integration. The noise in~\eqref{eq:mean_field_cl} is often called \emph{transport noise} (see e.g.~\cite{flandoli_gubinelli_priola}): If $X^\eps=X_t^\eps(x,s,\omega)$ is the (stochastic) flow generated by the family of SDEs
\begin{equation}\label{eq:stochastic_flow}
\begin{cases}
    dX_t^\eps = a\bigl(m^\eps(X_t^\eps,t)\bigr)\,dt + \eps\,dW_t &\text{for}\ s < t < T, \\
    X_s^\eps(x,s,\omega) = x
\end{cases}
\end{equation}
and we denote $X_t^\varepsilon(x) \coloneqq X_t^\varepsilon(x, 0, \omega)$, then (at least formally) the pushforward formula $u_t^\eps(\omega) \coloneqq X_t^\eps(\cdot,\omega)_\# u_\initial$ is the unique solution of~\eqref{eq:mean_field_cl}. See Theorem~\ref{thm:main_spde}\textit{(i)} for a precise formulation.

Note that taking the expected value of the SPDE~\eqref{eq:mean_field_cl} and applying Ito's formula, one observes (formally, for now) that $m^\eps$ solves the parabolic equation
\begin{equation} \label{eq:mean_field}
    \begin{cases}
        \partial_t m^\eps + \nabla \cdot f(m^\eps) = \tfrac{\eps^2}{2} \Delta m^\eps & \text{in}\ \R^d \times (0, T), \\
        m^\eps|_{t = 0} = u_\initial & \text{on } \R^d.
    \end{cases}
\end{equation}
Thus, being the solution of a parabolic PDE, the ``mean field'' $m^\eps$ in \eqref{eq:mean_field_cl} is actually quite regular. This partially motivates why one might hope for well-posedness and a well-behaved zero-noise limit for \eqref{eq:mean_field_cl}. We discuss further motivation for the specific formulation of \eqref{eq:mean_field_cl} in Section~\ref{sec:motivation-and-earlier-works}.

The following theorem establishes the existence and uniqueness of weak solutions to \eqref{eq:mean_field_cl}. For the precise definition of a weak solution, we refer the reader to Section~\ref{subsec:existence_uniqueness}.

\begin{theorem} \label{thm:main_spde}
    Let ${(\Omega, \F, \P)}$ be a probability space with a $d$-dimensional Brownian motion $(W_t)_{t \geq 0}$. Let $\eps > 0$ and $T > 0$. Assume that $u_\initial \in L^\infty(\mathbb{R}^d)$ and that $f \in C^{1, \alpha}_\loc(\mathbb{R}; \mathbb{R}^d)$ for some $\alpha \in (0, 1)$. Then:
    \begin{enumerate}[label=(\roman*)]
        \item The stochastic mean-field conservation law~\eqref{eq:mean_field_cl} has a unique weak solution $u^\eps$ lying $\Pas$ in $L^\infty(\R^d\times(0,T))$. This solution is given by the pushforward of the initial data along the flow,
        \begin{subequations}
        \begin{equation} \label{eq:pushforward-solution}
            u_t^\eps = (X_t^\eps)_\# u_\initial,
        \end{equation}
        that is, for all $t \in [0, T]$ and all $\vartheta \in C_c(\R^d)$, $\Pas$,
        \begin{equation} \label{eq:pushforward-solution-weak}
            \int_{\mathbb{R}^d} \vartheta(x) u_t^\eps(x) \,dx = \int_{\mathbb{R}^d} \vartheta(X_t^\eps(x)) u_\initial(x) \,dx,
        \end{equation}
        \end{subequations}
        where $X^\eps$ is the stochastic flow of \diffeomorphisms generated by the system of SDEs~\eqref{eq:stochastic_flow}.
        \item The process $u^\eps$ provides a stochastic representation for the solution of the viscous conservation law~\eqref{eq:mean_field}, in the sense that its mean, $m^\eps \coloneqq \E[u^\eps]$, is the unique weak solution of~\eqref{eq:mean_field}.
    \end{enumerate}
\end{theorem}

The next step is to verify that the stochastic formulation is consistent with the classical deterministic theory in the zero-noise limit. Our second main result confirms that, under certain conditions, the solutions $u^\eps$ converge to the entropy solution of the deterministic conservation law~\eqref{eq:cl} as $\eps \to 0$. Unfortunately, for technical reasons we were only able to prove this result in the one-dimensional, convex case.

\begin{theorem}\label{thm:main_zero_noise}
    Assume $d = 1$, that $u_\initial \in \BV_\loc \cap L^\infty \cap L^1(\R)$, and let the flux $f \in C^2(\R)$ satisfy $\inf f'' > 0$. For each $\eps > 0$, let $u^\eps$ be the unique weak solution of the stochastic mean-field conservation law ~\eqref{eq:mean_field_cl}, and let $u$ be the unique entropy solution of the corresponding deterministic conservation law~\eqref{eq:cl}. Then, in the zero-noise limit $\eps\to0$, $u^{\eps}$ converges to $u$, weak-$*$ in space, uniformly in time, and strongly in $L^p(\Omega)$ for all $p \geq 1$. More precisely, for any continuous and bounded test function $\vartheta \in C_b(\R)$,
    \begin{equation*} \label{eq:spde_convergence}
        \lim_{\eps \to 0} \E\biggl[\sup_{t \in [0, T]} \biggl|\int_{\R} \vartheta(x) \bigl(u_t^\eps(x) - u(x, t)\bigr)\, dx\biggr|^p \biggr] = 0.
    \end{equation*}
\end{theorem}

The specific mode of convergence presented in Theorem~\ref{thm:main_zero_noise} is discussed in Example~\ref{example:compressive}.

\subsection{Motivation and connection to earlier work}\label{sec:motivation-and-earlier-works}

The reason one might hope for well-posedness of~\eqref{eq:mean_field_cl}, and its convergence to the correct solution as $\eps\to0$, is threefold.

First, as already mentioned, the mean field $m^\eps$ satisfies the parabolic equation~\eqref{eq:mean_field}. Due to the smoothing properties of this equation, the drift $a(m^\eps)$ appearing in~\eqref{eq:mean_field_cl} and~\eqref{eq:stochastic_flow} possesses enhanced regularity. Second, the sequence $(m^\eps)_{\eps>0}$ converges to the entropy solution of~\eqref{eq:cl} as $\eps\to0$. We refer to \cite[Appendix~B]{holden_risebro_2015} and \cite{kruzkov_1970} for these well-established results.

Third, the specific form of the stochastic problem~\eqref{eq:mean_field_cl} is motivated by a recent result of ours~\cite{fjordholm_maehlen_oerke}: A weak solution $u$ of~\eqref{eq:cl} is the entropy solution if and only if the equation
\begin{equation} \label{eq:cl_flow}
    \begin{cases}
        \tfrac{d}{dt}{X}_t = a_k\bigl(u(X_t,t)\bigr) &\text{for}\ t > s, \\
        X_s = x
    \end{cases} \qquad \text{where } a_k(u)\coloneqq \begin{cases}
     \frac{f(u) - f(k)}{u - k} & (u\neq k) \\ f'(k) & (u=k)
    \end{cases}
\end{equation}
has a unique solution for all $k \in \R$ and every starting point $(x,s)$ (see Theorem~\ref{thm:particle_paths} for a precise formulation). Unfortunately, this result holds only in $d=1$ spatial dimensions, so our zero-noise result will be restricted to one dimension. Comparing to the stochastic problem~\eqref{eq:stochastic_flow}, and recalling that $\lim_{\eps\to0}m^\eps$ is the entropy solution of~\eqref{eq:cl}, one might hope that $\lim_{\eps\to0}X^\eps=X$, and as a consequence, that $\lim_{\eps\to0}u^\eps$ is the entropy solution of~\eqref{eq:cl}. This is indeed confirmed in Theorem~\ref{thm:main_zero_noise}.

The above motivation provides the main outline of our method of proof. The well-posedness of the SPDE \eqref{eq:mean_field_cl} is the content of Theorem~\ref{thm:main_spde}, and will be carried out in Section~\ref{sec:well_posedness}. Although somewhat technical, this part of the analysis mostly follows standard techniques.
Next, showing that the sequence of flows $(X^\eps)_{\eps>0}$ is precompact is easy, but identifying the equation satisfied by its limit is not. To get an idea of the problems that might arise when perturbing differential equations, consider the one-dimensional deterministic ordinary differential equation (ODE) $\frac{d}{dt}X^n_t = b^n(X^n_t,t)$, $X^n_0(x)=x$ (for $n\in\N$), where $b^n(x,t)\coloneqq (1-n|x-t|)\ind_{[t-1/n,t+1/n]}(x)$. These ODEs are well-posed, and the drifts $b^n$ converge in $L^\infty_tL^1_x$ to the trivial drift $b\equiv0$, but the corresponding flows $X^n$ converge to the non-trivial flow $X_t(x) = x + (t-x)\ind_{(0,t)}(x)$ (for $t\geq0$). (Note that $L^\infty_tL^1_x$ is precisely the type of convergence we can expect from our drifts $a(m^\eps)$.) In Section~\ref{sec:filippov} we describe a novel \emph{structural} condition ensuring that the zero-noise limit $\lim_{\eps\to0}X^\eps$ satisfies the ``reasonable'' ODE, and in Section~\ref{sec:zero-noise} we show that if $f$ is convex, then the structural condition from Section~\ref{sec:filippov} is indeed satisfied for \eqref{eq:stochastic_flow}---thus concluding the proof of Theorem~\ref{thm:main_zero_noise}.

The problem of zero-noise limits for differential equations is old, but highly non-trivial, and as such, very few results are available. The classical reference for ODEs is Bafico--Baldi~\cite{bafico_baldi_1982}, who identified the zero-noise limit of certain SDEs for which the corresponding deterministic problem is ill-posed. For PDEs, only a few results are available. We mention in particular Attanasio--Flandoli~\cite{attanasio_flandoli_2009}, who study the zero-noise limit for certain linear transport equations; in particular, they contrast the zero-noise limit to the zero-viscosity limit, which are two related, but very different types of approximations.

The form of the stochastically perturbed equation \eqref{eq:mean_field_cl} and \eqref{eq:stochastic_flow} was motivated by McKean--Vlasov equations, such as those appearing in mean-field games; see e.g.\ Cardaliaguet~\cite{cardaliaguet_notes_2013}, Barbu--R\"ockner~\cite{barbuNonlinearFokkerPlanck2020} and references therein. Equations similar to \eqref{eq:mean_field_cl} have appeared in works on turbulence modeling by Eyink, Drivas, Holm, Leahy and others; see e.g.~\cite{drivas_holm_leahy,eyink_spontaneous_2015} and references therein. The so-called \emph{LA SALT} version of the Burgers equation (i.e.,~\eqref{eq:cl} with $f(u)\coloneqq u^2/2$) can be written as
\begin{equation} \label{eq:la-salt-burgers}
    \begin{cases}
    dv^\eps + m^\eps\partial_x v^\eps\, dt + \eps \nabla v^\eps \circ dW_t = 0 &\text{in}\ \R \times (0, T),\\
    m^\eps \coloneqq \E[v^\eps], \\
    v^\eps\bigr|_{t = 0} = u_\initial & \text{on } \R
    \end{cases}
\end{equation}
(see~\cite[Sec.~4.2]{drivas_holm_leahy}). Just as for our equation \eqref{eq:mean_field_cl}, the mean field $m^\eps$ in~\eqref{eq:la-salt-burgers} will solve the viscous Burgers equation. In contrast to the continuity equation~\eqref{eq:mean_field_cl}, however, \eqref{eq:la-salt-burgers} is a \emph{transport equation}, and therefore one would expect the solution of~\eqref{eq:la-salt-burgers} to be given by
\[
v^\eps(x,t) = u_\initial\bigl(Y^\eps_0(x,t)\bigr),
\]
where
\begin{equation}\label{eq:la-salt-characteristics}
\begin{cases}
    dY^\eps_t = m^\eps(Y^\eps_t,t)\,dt + \eps\,dW_t & \text{for } 0<t<s, \\
    Y^\eps_s(x,s,\omega) = x
\end{cases}
\end{equation}
(this is indeed the case if, for example, $u_\initial \in \BV_\loc(\R^d)$; see \cite{oerke_2025}).

Contrast this to our model \eqref{eq:mean_field_cl}--\eqref{eq:mean_field}, which for Burgers' equation reads as
\begin{equation}\label{eq:mean_field_cl-burgers}\tag{\ref*{eq:mean_field_cl}'}
    \begin{cases}
    du^\eps + \partial_x\bigl(\tfrac{m^\eps}{2} u^\eps\bigr)\, dt + \eps \nabla u^\eps \circ dW_t = 0 &\text{in}\ \R \times (0, T),\\
    m^\eps \coloneqq \E[u^\eps], \\
    u^\eps\bigr|_{t = 0} = u_\initial & \text{on } \R,
    \end{cases}
\end{equation}
whose solution is given by the pushforward formula $u_t^\eps = (X^\eps_t(\cdot,0))_\#u_\initial$ (see~\eqref{eq:pushforward-solution}),
where
\begin{equation}\label{eq:stochastic_flow_burgers}\tag{\ref*{eq:stochastic_flow}'}
\begin{cases}
    dX^\eps_t = \dfrac{m^\eps(X^\eps_t,t)}{2}\,dt + \eps\,dW_t & \text{for } s<t<T, \\
    X^\eps_s(x,s,\omega) = x.
\end{cases}
\end{equation}
Note the crucial factor $\frac12$ in~\eqref{eq:stochastic_flow_burgers}, which is not present in~\eqref{eq:la-salt-characteristics}. We compare these two models both in the viscous ($\eps>0$) and inviscid ($\eps=0$) regimes in the following two examples.

\begin{example} \label{example:compressive}
We compare the behavior of the one-dimensional stochastic mean-field Burgers equation~\eqref{eq:mean_field_cl-burgers} and the \emph{LA SALT} Burgers equation~\eqref{eq:la-salt-burgers}, both for $\eps = 0$ and $\eps > 0$. In this example we study the ``compressive'' initial data
\begin{equation*}
    u_\initial(x) =
    \begin{cases}
        1 & \text{if } x < 0, \\
        -1 & \text{else}.
    \end{cases}
\end{equation*}

The velocity fields $m^\eps/2$ in~\eqref{eq:mean_field_cl-burgers} and $m^\eps$ in~\eqref{eq:la-salt-burgers} are given by the unique solution of the one-dimensional Burgers equation
\begin{equation} \label{eq:burgers_example}
    \partial_t m^\eps + \partial_x\biggl(\frac{1}{2}(m^\eps)^2\biggr) = \frac{\eps^2}{2} \partial_{xx} m^\eps
\end{equation}
with initial data $u_\initial$. In the inviscid case ($\eps = 0$), the equation becomes the hyperbolic conservation law $\partial_t u + \partial_x(u^2/2) = 0$, whose unique entropy solution is the stationary shock wave $u(t) \equiv u_\initial$. In the viscous case ($\eps > 0$), the diffusion term $\frac{\eps^2}{2} \partial_{xx} m^\eps$ smoothens out the initial discontinuity for all $t > 0$, and the resulting solution $m^\eps$ is a smooth profile that converges to $u$ in $L^1_\loc(\mathbb{R} \times (0, T))$ as $\eps \to 0$.

Starting with the stochastic mean-field Burgers equation \eqref{eq:mean_field_cl-burgers}, we now use these two fixed solutions---the shock $u$ and the smooth profile $m^\eps$---as velocity fields corresponding to the inviscid case ($\eps = 0$) and the viscous case ($\eps > 0$):
\begin{equation*}
    \partial_t u + \partial_x \biggl(\frac{1}{2} u^2\biggr) = 0, \qquad d u^\eps + \partial_x \biggl(\frac{m^\eps}{2} u^\eps\biggr)\, dt + \eps \partial_x u^\eps\circ dW_t = 0.
\end{equation*}
As established in \cite{fjordholm_maehlen_oerke} (see Theorem~\ref{thm:particle_paths} in Section~\ref{sec:prelims}) and Theorem~\ref{thm:main_spde}, these equations both have unique weak solutions given by
\begin{equation} \label{eq:compressive_solution_formulas}
    u(\cdot, t) = (X_t)_{\#} u_\initial, \qquad u_t^\eps = (X_t^\eps)_{\#} u_\initial,
\end{equation}
where the flows are generated by \eqref{eq:stochastic_flow_burgers}, $X$ being the (Filippov) flow for the velocity field $u/2$, and $X^\eps$ the stochastic flow for velocity field $m^\eps/2$.

Although the similar solution formulas in \eqref{eq:compressive_solution_formulas} hint to a connection in the zero-noise limit, the qualitative behavior of these flows is drastically different. In the deterministic equation (the inviscid Burgers equation), the discontinuous velocity field $u/2$ causes particle paths to merge at the shock front. This phenomenon allows for cancellation of mass carried by the entropy solution $u$. In contrast, the smooth velocity field $m^\eps/2$ in the stochastic equation on the right generates a stochastic flow that is a \diffeomorphism, where paths can never merge. Consequently, there is no mechanism for mass cancellation for the solution $u^\eps$.

Nonetheless, the solutions $u^\eps$ will converge to the entropy solution $u$, although the aforementioned distinction between $u$ and $m^\eps$ has consequences for the convergence $u^\eps \to u$. Although the mean field $m^\varepsilon$ converges to $u$ in $L^1$, the lack of a mass cancellation mechanism means that mass of $u^\eps$ of opposite signs piles up on both sides of the shock. This situation is shown in Figure~\ref{fig:compressive_convergence_figure_mean_field_cl} (see Appendix~\ref{app:numerical_methods} and Table~\ref{table:default_params} for a description of numerical methods and default run parameters). This prevents $L^1$-convergence of $u^\eps$ to the deterministic entropy solution $u$, and necessitates the spatial weak-$*$ convergence established in Theorem~\ref{thm:main_zero_noise}.

\begin{figure}[htbp]
    \centering
    \includegraphics[width=0.8\textwidth]{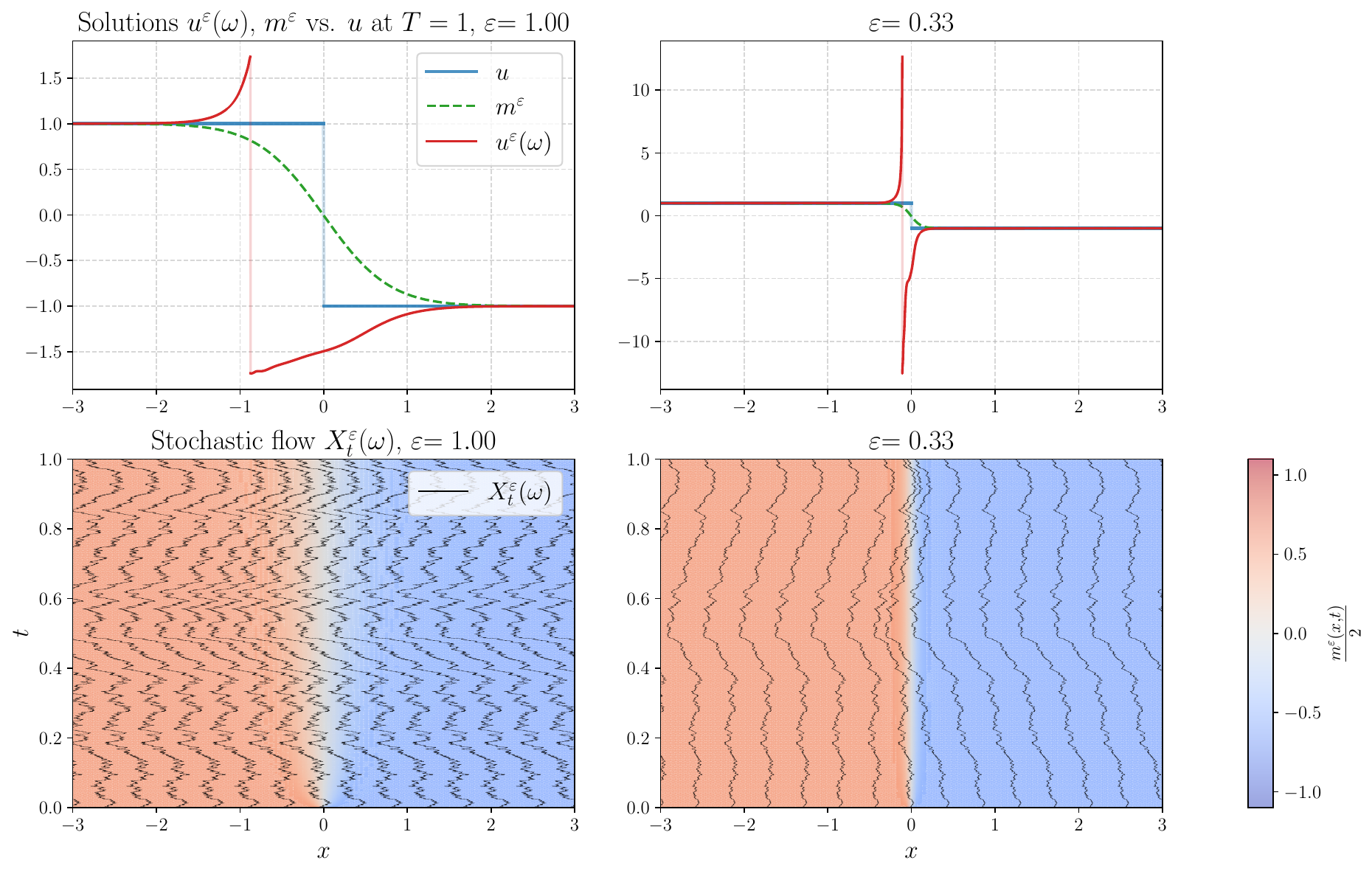}
    \caption{A sample path $u^\eps(\omega)$ of the stochastic mean-field Burgers equation~\eqref{eq:mean_field_cl-burgers} with compressive Riemann initial data, for $\eps = 1$ (left) and $\eps = 1/3$ (right). The top panels compare $u^\eps(\omega)$ to the viscous solution $m^\eps$ of~\eqref{eq:burgers_example} and the entropy solution $u$. The bottom panels show the stochastic flow $X^\eps(\omega)$, generated by the SDE \eqref{eq:stochastic_flow_burgers} with drift $m^\eps/2$, and used to construct $u^\eps(\omega)$ via the pushforward formula $u^\eps = X^\eps_\# u_\initial$. In the flow plots, the background is shaded proportional to the magnitude of the drift~$m^\eps/2$.}
    \label{fig:compressive_convergence_figure_mean_field_cl}
\end{figure}

Note that while individual realizations of the stochastic solution $u^\eps$ accumulate mass on both sides of the shock, thereby amplifying the discontinuity, its expectation, $\E[u^\eps]$, recovers the smooth deterministic profile $m^\eps$. This is shown in Figure~\ref{fig:expectation} (left panel).

Next, we consider the LA SALT Burgers equation~\eqref{eq:la-salt-burgers}, using the entropy solution $u \equiv u_\initial$ and the viscous solution $m^\eps$ of \eqref{eq:burgers_example} as velocity fields for the two transport equations \eqref{eq:la-salt-burgers} corresponding to the inviscid case ($\eps = 0$) and the viscous case ($\eps > 0$):
\begin{equation*}
    \partial_t u + u \partial_x u = 0, \qquad d v^\eps + m^\eps \partial_x v^\eps\, dt + \eps \partial_x v^\eps \circ dW_t = 0.
\end{equation*}
The stochastic equation on the right has a unique weak solution given by $v^\varepsilon(x, t) = u_\initial(Y^\eps_0(x, t))$, where $Y^\eps$ is the stochastic flow generated by \eqref{eq:la-salt-characteristics}. This solution is shown in Figure~\ref{fig:compressive_convergence_figure_la_salt} for two different values of $\eps$.

\begin{figure}[htbp]
    \centering
    \includegraphics[width=0.8\textwidth]{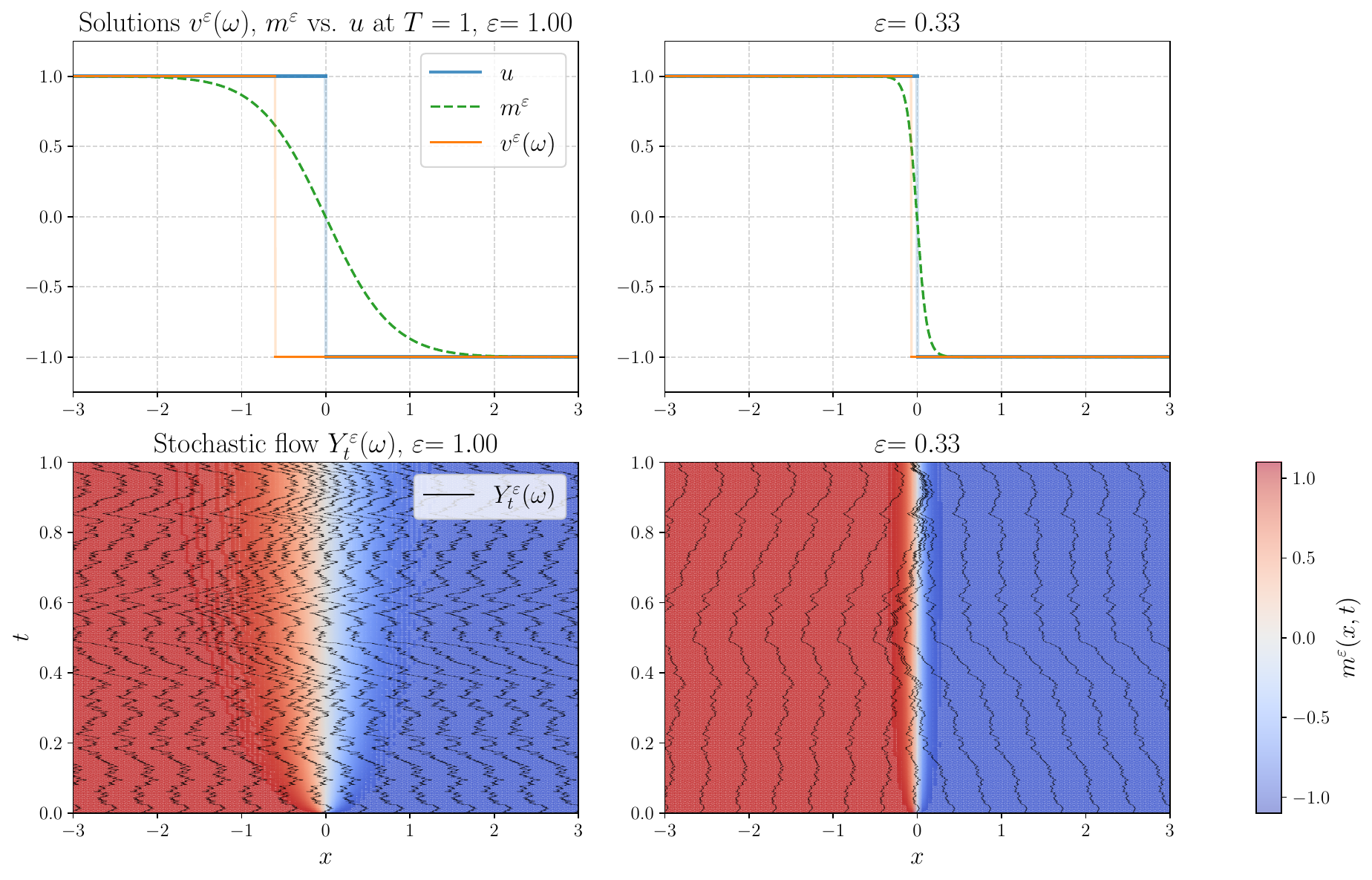}
    \caption{A sample path $v^\eps(\omega)$ of the LA SALT Burgers equation~\eqref{eq:la-salt-burgers} with compressive Riemann initial data, for $\eps = 1$ (left) and $\eps = 1/3$ (right). The top panels compare $v^\eps(\omega)$ to the viscous solution $m^\eps$ of~\eqref{eq:burgers_example} and the entropy solution $u$. The bottom panels show the stochastic flow $Y^\eps(\omega)$, generated by the SDE \eqref{eq:la-salt-characteristics} with drift  $m^\eps$, and used to construct $v^\eps(\omega)$ via the formula $v^\eps = u_\initial\bigl(Y^\eps_0(x,t)\bigr)$. The background is shaded proportional to the magnitude of the drift $m^\eps$. The driving Brownian motion is identical to that used in Figure~\ref{fig:compressive_convergence_figure_mean_field_cl}.}
    \label{fig:compressive_convergence_figure_la_salt}
\end{figure}
\end{example}

\begin{figure}[htbp]
    \centering
    \includegraphics[width=0.4\textwidth]{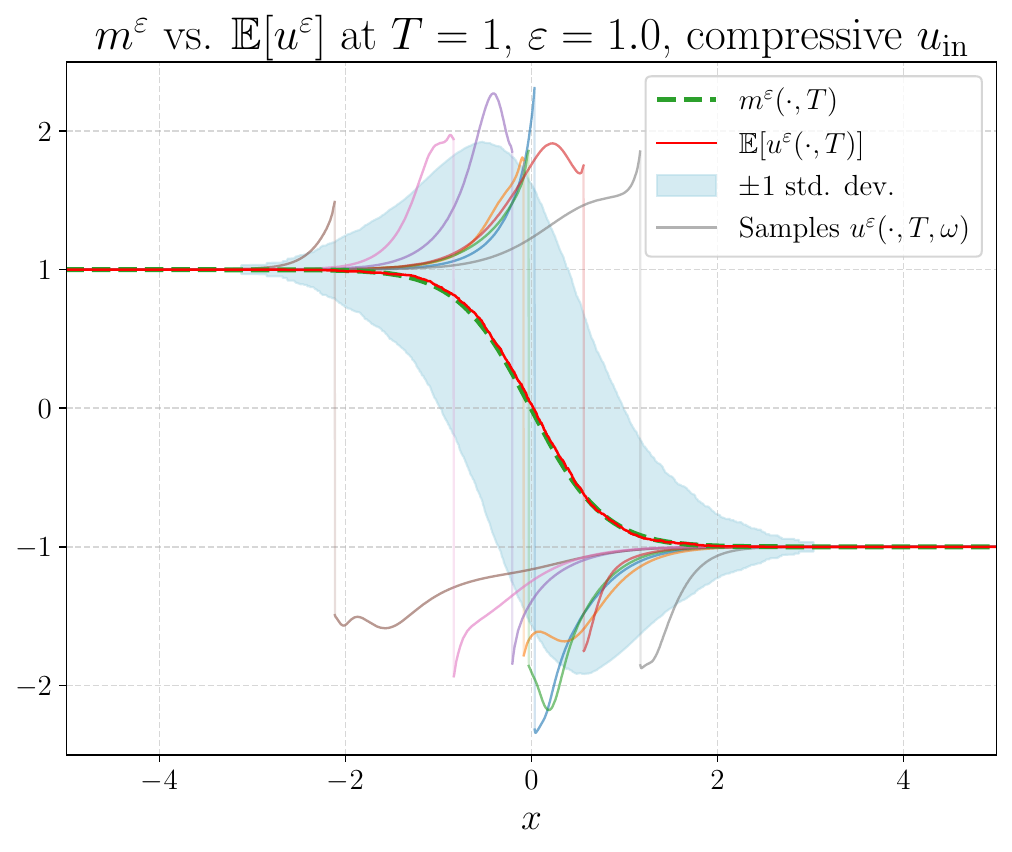}
    \includegraphics[width=0.4\textwidth]{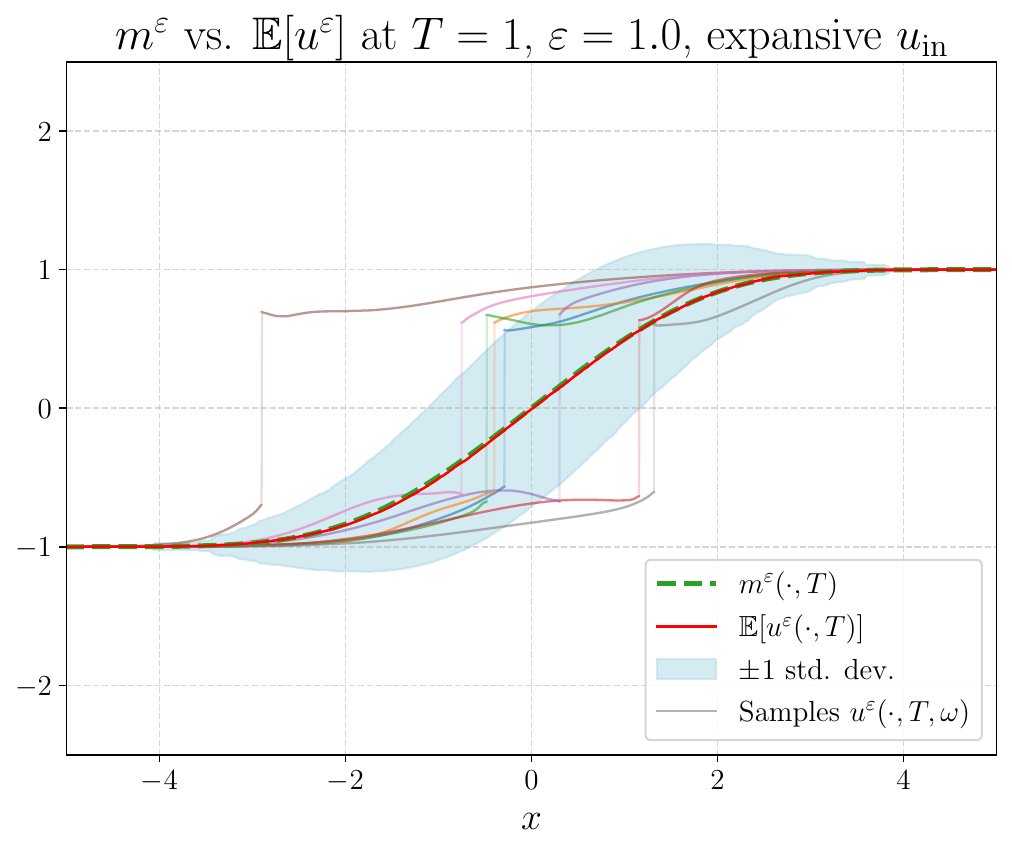}
    \caption{Comparison between the deterministic viscous solution $m^\eps$ of \eqref{eq:burgers_example} and the sample mean of the stochastic mean-field Burgers equation \eqref{eq:mean_field_cl-burgers}, for $T = 1$ and $\eps = 1$, in the compressive case from Example~\ref{example:compressive} (left) and the expansive case from Example~\ref{example:expansive} (right). The dashed green line shows the viscous solution $m^\eps(x, T)$, and the solid red line shows the sample mean, computed from $N=5000$ Monte Carlo simulations. The shaded blue region represents one standard deviation around the mean, indicating the typical spread of the stochastic solutions (a few of which are plotted in faint lines).}
    \label{fig:expectation}
\end{figure}

\begin{example} \label{example:expansive}
In this example, we consider the ``expansive'' initial data
\begin{equation*}
    u_\initial(x) =
    \begin{cases}
        -1 & \text{if } x < 0, \\
        1 & \text{else},
    \end{cases}
\end{equation*}
again for the one-dimensional stochastic mean-field Burgers equation~\eqref{eq:mean_field_cl-burgers} and the \emph{LA SALT} Burgers equation~\eqref{eq:la-salt-burgers}.

In this case, the entropy solution $u$ of the inviscid Burgers equation (equation~\eqref{eq:burgers_example} with $\eps = 0$) is the rarefaction wave
\begin{equation*}
    u(x, t) =
    \begin{cases}
        -1 & \text{if } x < -t, \\
        x/t & \text{if } -t < x < t, \\
        1 & \text{if } t < x,
    \end{cases}
\end{equation*}
and the viscous solution $m^\eps$ is a smooth approximation of $u$. Following Example~\ref{example:compressive}, we use these as velocity fields for~\eqref{eq:mean_field_cl-burgers} and ~\eqref{eq:la-salt-burgers}, respectively. This yields two distinct stochastic solutions: $u^\eps$, given by the pushforward $u_t^\eps = (X_t^\eps)_{\#} u_\initial$ along the flow $X^\eps$ of \eqref{eq:stochastic_flow_burgers}, and $v^\eps$, given by the composition $v^\eps_t(x) = u_\initial(Y^\eps_0(x, t))$ with the backward flow $Y^\eps_0(\cdot, t)$ of \eqref{eq:la-salt-characteristics}.

These solutions are illustrated in Figures~\ref{fig:expansive_convergence_figure_mean_field_cl} and~\ref{fig:expansive_convergence_figure_la_salt} for two different values of $\eps$.

\begin{figure}[htbp]
    \centering
    \includegraphics[width=0.8\textwidth]{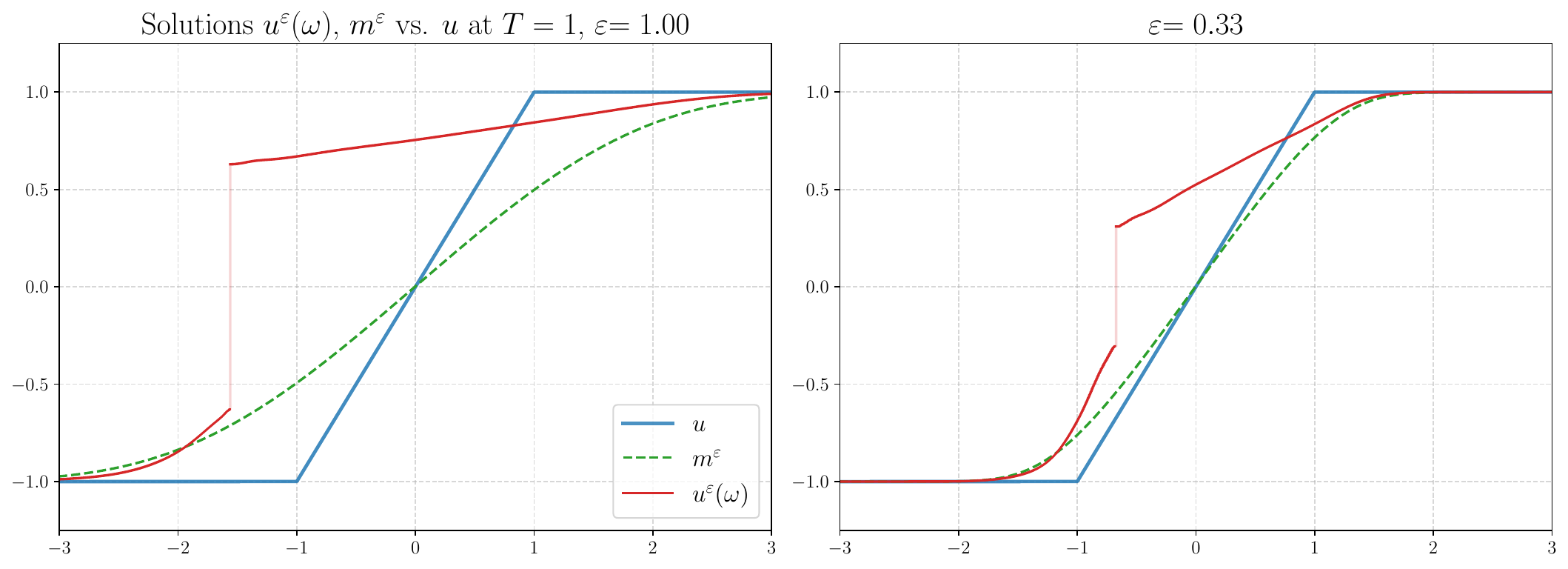}
    \caption{A sample path $u^\eps(\omega)$ of~\eqref{eq:mean_field_cl-burgers} with expansive Riemann initial data, plotted against the viscous solution $m^\eps$ of~\eqref{eq:burgers_example} and the entropy solution $u$, for $\eps = 1$ (left) and $\eps = 1/3$ (right).}
    \label{fig:expansive_convergence_figure_mean_field_cl}
\end{figure}

\begin{figure}[htbp]
    \centering
    \includegraphics[width=0.8\textwidth]{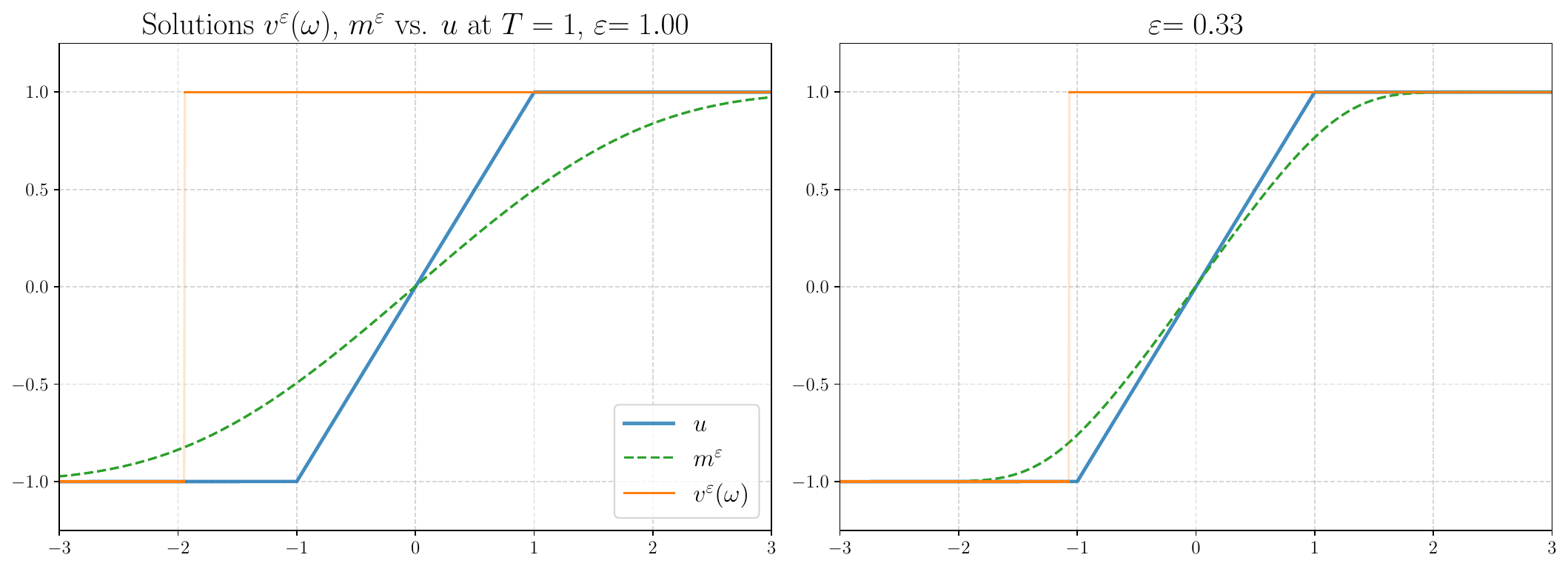}
    \caption{A sample path $v^\eps(\omega)$ of~\eqref{eq:la-salt-burgers} with expansive Riemann initial data, plotted against the viscous solution $m^\eps$ of~\eqref{eq:burgers_example} and the entropy solution $u$, for $\eps = 1$ (left) and $\eps = 1/3$ (right).}
    \label{fig:expansive_convergence_figure_la_salt}
\end{figure}

The failure of the LA SALT Burgers equation is now apparent. Its solution $v^\eps$ is structurally constrained to the form $v^\eps(x, t) = u_\initial(Y^\eps_0(x, t))$, which merely applies a random mapping to the initial discontinuity (Figure~\ref{fig:expansive_convergence_figure_la_salt}). Consequently, this solution is incapable of approximating the rarefaction wave in the zero-noise limit. Our stochastic approximation $u^\eps$, on the other hand, shows convergent behavior as $\eps$ decreases (Figure~\ref{fig:expansive_convergence_figure_mean_field_cl}).
\end{example}

\section{Preliminaries} \label{sec:prelims}

In this section, we collect definitions, notations, and basic results that will be used throughout the paper.

\subsection{Scalar conservation laws and particle paths} \label{subsec:particle_paths}

A function $u\in L^\infty(\R^d\times(0, T))$ is a \emph{weak solution} of the scalar conservation law~\eqref{eq:cl} if
\begin{equation*}
    \int_{0}^T \int_{\R^d} u\partial_t\varphi + f(u)\cdot \nabla \varphi\,dx\,dt + \int_{\R^d} u_\initial(x)\varphi(x,0)\,dx = 0
\end{equation*}
for all $\varphi\in C_c^\infty(\R^d \times [0, T))$. We say that a pair of functions $(\eta,q)$ is an entropy pair if $\eta\colon\R\to\R$ is convex and $q\colon\R\to\R^d$ satisfies $q' = \eta'f'$. A weak solution $u$ is an \emph{entropy solution} if
\begin{equation*}
    \int_0^T \int_{\R^d} \eta(u)\partial_t\varphi + q(u)\cdot \nabla \varphi\,dx\,dt + \int_{\R^d} \eta(u_\initial(x))\varphi(x,0)\,dx \geq 0
\end{equation*}
for all nonnegative $\varphi \in C_c^\infty(\R^d\times[0, T))$ and all entropy pairs $(\eta, q)$. Kruzkhov~\cite{kruzkov_1970} proved that there exists a unique entropy solution of~\eqref{eq:cl} for any $u_\initial \in L^\infty(\R^d)$. See Kruzkhov~\cite{kruzkov_1970} and the monographs by Holden~and~Risebro~\cite{holden_risebro_2015} and Dafermos~\cite{dafermos_2016} for the general theory of hyperbolic conservation laws.

In one spatial dimension, if $u_\initial\in \BV_\loc(\R)$, then also $u(t) \in \BV_\loc(\R)$ for all $t>0$. Moreover, there is a special relationship between the entropy solution and the particle paths, captured by the following theorem.

\begin{theorem}[Fjordholm, Mæhlen, Ørke~\cite{fjordholm_maehlen_oerke}] \label{thm:particle_paths}
    Let $f\in C^1(\R)$ and $u_\initial\in \BV_{\loc}\cap L^\infty(\R)$. If ${u\in L^\infty(\R\times\R_+)}$ is a weak solution of~\eqref{eq:cl} with $u(t) \in \BV_{\loc}(\R)$ for a.e.~$t\geq 0$, then the following are equivalent:
    \begin{enumerate}[label=(\roman*)]
    \item $u$ is the entropy solution of~\eqref{eq:cl}.
    \item The ODE
    \begin{equation} \label{eq:flow_with_k}
        \begin{cases}
            \tfrac{d}{dt}{X}_t = a_k\bigl(u(X_t,t)\bigr) &\text{for}\ t > s, \\
            X_s = x
        \end{cases} \qquad \text{where } a_k(u) \coloneqq \begin{cases}
             \frac{f(u) - f(k)}{u - k} & (u\neq k) \\ f'(k) & (u=k)
        \end{cases}
    \end{equation}
    is well-posed in the Filippov sense (see Section \ref{subsec:filippov_solutions}) for all $x \in \R$, $s \geq 0$ and all $k \in \R$.
    \end{enumerate}
    Moreover, for any $k \in \R$, the entropy solution $u$ satisfies $u(t) = k + (X^k_t)_\#(u_\initial - k)$ for all $t \geq 0$, i.e.
    \begin{equation*}\label{eq:representation_formula}
         \int_{\R} \vartheta(x) u(x, t)\, dx = k \int_{\R} \vartheta(x)\, dx + \int_{\R} \vartheta\bigl(X_t^k(x)\bigr) (u_\initial(x) - k)\, dx
    \end{equation*}
    for all $\vartheta \in C_c(\R)$ and $t \geq 0$, where $X^k=X^k_t(x)$ is the unique Filippov flow of the ODE~\eqref{eq:flow_with_k}.
\end{theorem}

Filippov solutions generalize the standard solution concept for equations with discontinuous right-hand sides; see Section~\ref{sec:filippov} and Filippov's original work~\cite{Filippov60}. For an example of an application of Theorem~\ref{thm:particle_paths}, we refer the reader to our work in~\cite{oerke_numerics_2025} on the formulation and analysis of a numerical scheme for scalar conservation laws.

\subsection{Stochastic flows of diffeomorphisms}\label{subsec:stoch_flows}

Following the foundational work of Kunita~\cite{kunita, kunita_1990}, we give the following definition.

\begin{definition}[Kunita~\cite{kunita_1990}] \label{def:stoch_flow}
    Let $X = X_{s, t}(x, \omega)$, defined for $s, t \in [0, T]$ and $x \in \R^d$, be a continuous $\R^d$-valued random field on a probability space $(\Omega, \F, \P)$. It is a \emph{stochastic flow of homeomorphisms} if for $\P$-a.e.~$\omega \in \Omega$, the family $(X_{s, t}(\omega))_{s, t \in [0, T]}$ is a flow of homeomorphisms on $\R^d$, i.e.
    \begin{enumerate}[label=(\roman*)]
        \item $X_{s, t}(\omega) = X_{r, t}(\omega) \circ X_{s, r}(\omega)$ for all $s,r, t \in [0, T]$,
        \item $X_{s, s}(\omega, x) = x$ for all $s \in [0, T]$ and $x \in \R^d$,
        \item $X_{s, t}(\omega)\colon \R^d \to \R^d$ is a homeomorphism on $\R^d$ for all $s, t \in [0, T]$.
    \end{enumerate}
    We say that $X$ is a \emph{stochastic flow of $C^k$-\diffeomorphisms} if
    \begin{enumerate}[resume*]
        \item $X_{s, t}(\omega)$ is $k$ times differentiable w.r.t.~$x$, and the derivatives are continuous in $(x, s, t)$.
    \end{enumerate}
    If in addition to (iv) the derivatives are Hölder continuous with exponent $\beta$ with respect to $x$, we say that $X$ is a \emph{stochastic flow of $C^{k, \beta}$-\diffeomorphisms}.
\end{definition}

We define the forward flow as the restriction of $X_{s, t}$ to the forward temporal indices $0 \leq s \leq t \leq T$. Such a flow is said to be generated by the SDE
\begin{equation} \label{eq:sde}
    \begin{cases}
        dX_t = b(X_t, t)\, dt + dW_t &\text{for}\ t > s, \\
        X_s = x,
    \end{cases}
\end{equation}
if it is a modification of the corresponding family of solutions $(X_t^{x, s} \colon 0 \leq s \leq t \leq T,\ x \in \R^d)$. Here, each solution $X_t^{x, s}$ is adapted to the underlying filtration $\F_{s, t}$, which we take to be the completed $\sigma$-algebra generated by $(W_u - W_r)_{s \leq r \leq u \leq t}$.

It is a classical result by Kunita~\cite{kunita, kunita_1990} that SDEs with regular drifts generate stochastic flows of \diffeomorphisms. We will require lower regularity than in Kunita's work, in particular drifts belonging to the space $L^q((0, T); C^{0, \alpha}(\R^d))$, defined as the closure of $C^\infty_c((0, T); C^{0, \alpha}(\R^d))$ in the norm
\begin{equation*}
    \norm{b}_{L^q((0, T); C^{0, \alpha}(\R^d))} = \biggl(\int_0^T \norm{b(t)}_{C^{0, \alpha}(\R^d)}^q\, dt\biggr)^{\frac{1}{q}}
\end{equation*}
(see~\cite{krylov_cz_2002} for more information about these spaces). We will need the following result.

\begin{theorem}[Ørke~\cite{oerke_2025}]\label{thm:well_posed_sde}
     Let $b \in L^q\bigl((0, T); C^{0, \alpha}(\R^d; \R^d)\bigr)$ for some $q \in [2, \infty)$ and $\alpha \in (0, 1)$. Then for any $x \in \R^d$ and $s \in [0, T]$, the SDE~\eqref{eq:sde} has a unique (strong) solution $(X_t^{x, s})_{t \in [s, T]}$. Moreover, the family of solutions ${(X_t^{x, s}\colon 0 \leq s \leq t \leq T,\, x \in \R^d)}$ has a modification, denoted by $X = X_{s, t}(x, \omega)$, which for all $\beta < \alpha$ is a forward stochastic flow of $C^{1, \beta}$-\diffeomorphisms on $\R^d$.
\end{theorem}

\subsection{Linear stochastic continuity equations}

On a probability space ${(\Omega, \F, \P)}$ with a $d$-dimensional Brownian motion $(W_t)_{t \geq 0}$ with respect to a given complete and right-continuous filtration $(\F_t)_{t \geq 0}$, consider the linear stochastic continuity equation
\begin{equation} \label{eq:stoch_continuity_linear}
\begin{cases}
    du + \nabla \cdot (b u)\, dt + \nabla u \circ dW_t = 0 &\text{in}\ \R^d \times (0, T),\\
    u|_{t = 0} = u_\initial & \text{on}\ \R^d.
\end{cases}
\end{equation}
As is often the case, this linear model is a foundational component for solving the more complex nonlinear equations presented in the introduction. Weak solutions will be understood in the following sense.

\begin{definition} \label{def:stoch_cont_weak}
    Let $u_\initial \in L^\infty(\R^d)$. A \emph{weak solution} of~\eqref{eq:stoch_continuity_linear} is a random field $u$ lying $\Pas$ in $L^\infty(\R^d\times(0,T))$, such that for all $\vartheta \in C^\infty_c(\R^d)$ the stochastic process ${t \mapsto \int_{\R^d} \vartheta(x)u_t(x)\, dx}$ is a continuous $\F_t$-\semimartingale and satisfies
    \begin{equation*}
        \int_{\R^d} \vartheta u_t\, dx = \int_{\R^d} \vartheta u_\initial\, dx + \int_0^t \int_{\R^d} b(\cdot, r) u_r \cdot \nabla \vartheta\, dx\,dr + \int_0^t \biggl(\int_{\R^d} u_r \nabla \vartheta\, dx \biggr) \circ dW_r
    \end{equation*}
    for all $t \in [0, T]$, $\Pas$.
\end{definition}

\begin{remark}
    The continuity of the process $t \mapsto \int_{\R^d} \vartheta(x)u_t(x)\, dx$ is to be understood as the existence of a version with almost surely continuous sample paths for every $\vartheta \in C_c^\infty(\R^d)$.
\end{remark}

Of course, the definition of weak solutions can be written equivalently in Itô integral form.

\begin{lemma} \label{lemma:system_ito_form}
    Let $u_\initial \in L^\infty(\R^d)$. A process $u$ lying $\Pas$ in $L^\infty(\R^d\times(0,T))$ is a weak solution of~\eqref{eq:stoch_continuity_linear} if and only if for all $\vartheta \in C^\infty_c(\R^d)$, the process ${t \mapsto \int_{\R^d} \vartheta(x)u_t(x)\, dx}$ is continuous, $\F_t$-adapted, and satisfies
    \begin{equation} \label{eq:mean_field_cl_weak_ito}
        \begin{aligned}
            \int_{\R^d} \vartheta u_t\, dx & = \int_{\R^d} \vartheta u_\initial\, dx + \int_0^t \int_{\R^d} b(\cdot, r) u_r \cdot \nabla \vartheta\, dx\,dr \\
            & \quad + \int_0^t \biggl(\int_{\R^d} u_r \nabla \vartheta\, dx \biggr)\cdot dW_r + \frac{1}{2} \int_0^t \int_{\R^d} u_r \Delta \vartheta\, dx\,dr
        \end{aligned}
    \end{equation}
    for all $t \in [0, T]$, $\Pas$.
\end{lemma}

The assumptions of Theorem~\ref{thm:well_posed_sde}, which ensure the existence of a stochastic flow of \diffeomorphisms, also allow the stochastic continuity equation~\eqref{eq:stoch_continuity_linear} to be solved via the method of characteristics.

\begin{theorem}
\label{thm:linear_continuity}
    Let $T > 0$, let $u_\initial \in L^\infty(\R^d)$, and assume that the velocity field $b$ belongs to $L^q\bigl((0, T); C^{0, \alpha}(\R^d; \R^d)\bigr)$ for some $q \in [2, \infty)$ and $\alpha \in (0, 1)$. Then the linear stochastic continuity equation~\eqref{eq:stoch_continuity_linear} has a unique weak solution given by $u_t = (X_t)_{\#} u_\initial$, that is,
    \begin{equation*}
    \begin{aligned}
        \int_{\R^d} \vartheta(x) u_t(x)\, dx & = \int_{\R^d} \vartheta(X_t(x)) u_\initial(x)\, dx \\
        & = \int_{\R^d} \vartheta(y) u_\initial\bigl(X_t^{-1}(y)\bigr) \det\bigl(\nabla X_t^{-1}(y) \bigr)\, dy
    \end{aligned}
    \end{equation*}
    for all $\vartheta \in C^\infty_c(\R^d)$ and for all $t \in [0, T]$, $\Pas$, where $X$ is the stochastic flow of \diffeomorphisms generated by~\eqref{eq:sde}.
\end{theorem}
\begin{proof}
By \cite[Theorem~1.3]{oerke_2025}, $u_t\coloneqq (X_t)_\#u_\initial$ is a weak solution of \eqref{eq:stoch_continuity_linear}, and for every $t\geq0$, $u_t$ is $\Pas$ a locally bounded Radon measure. By the equivalent formulation $u_t= (X_t)_\#u_\initial = \det(\nabla_x X_t^{-1}) u_\initial\circ X_t^{-1}$ (see e.g.~\cite[eq.~2.5]{ambrosioContinuityEquationsODE2014}), and the fact that $\nabla_x X_t^{-1}$ is uniformly bounded for all $t\in[0,T]$, $\Pas$ (cf.~\cite[Theorem~1.1~(ii)]{oerke_2025}), we conclude that $u$ lies in $L^\infty(\R^d\times(0,T))$, $\Pas$.
\end{proof}

\section{The stochastic mean-field conservation law} \label{sec:well_posedness}

This section is dedicated to proving the well-posedness of the main problem~\eqref{eq:mean_field_cl}, restated here for convenience:
\begin{equation} \label{eq:mean_field_cl_2}
    \begin{cases}
        du + \nabla \cdot \bigl(a(m) u\bigr)\, dt + \eps \nabla u \circ dW_t = 0,\\
        m = \E[u], \\
        u\bigr|_{t = 0} = u_\initial.
    \end{cases}
\end{equation}
The parameter $\eps > 0$ is fixed throughout this section, so we omit it from the superscript. Fix also a time horizon $T > 0$.

\begin{definition}
\label{def:mean_field_weak_solution}
A \emph{weak solution} of \eqref{eq:mean_field_cl_2} is a weak solution
in the pathwise sense of Definition~\ref{def:stoch_cont_weak},
with drift $b\coloneqq a(m)$ and noise coefficient $\varepsilon$, which additionally satisfies $m\in L^\infty(\mathbb R^d\times(0,T))$ and, for every compact set $K \subset \R^d$,
\begin{equation}\label{eq:cl-weak-soln-integrability}
\E \int_0^T \norm{u_t}_{L^1(K)}^2\,dt < \infty.
\end{equation}
More precisely, for every
$\vartheta\in C_c^\infty(\mathbb R^d)$ and $t\in(0,T)$,
\begin{equation} \label{eq:cl_mean_field_weak}
    \begin{aligned}
        \int_{\R^d} \vartheta u_t\, dx & = \int_{\R^d} \vartheta u_\initial\, dx + \int_0^t \int_{\R^d} a(\E[u_r]) u_r \cdot \nabla \vartheta\, dx\,dr \\
        & \quad + \eps \int_0^t \biggl(\int_{\R^d} u_r \nabla \vartheta\, dx \biggr) \circ dW_r.
    \end{aligned}
\end{equation}
\end{definition}

Our argument proceeds in two main steps. First, we formally take the expectation of~\eqref{eq:mean_field_cl_2} to derive a deterministic parabolic PDE for the mean field $m$:
\begin{equation} \label{eq:mean_field_2}
    \begin{cases}
        \partial_t m + \nabla \cdot f(m) = \tfrac{\eps^2}{2} \Delta m, \\
        m|_{t = 0} = u_\initial
    \end{cases}
\end{equation}
(recall that $a(v) \coloneqq \frac{f(v)-f(0)}{v}$). In Section~\ref{subsec:mean_field}, we show that this equation admits a unique solution and establish some crucial regularity properties. Second, in Section~\ref{subsec:existence_uniqueness}, we view~\eqref{eq:mean_field_cl_2} as a linear stochastic continuity equation where the velocity field $a(m)$ is fixed, and show its well-posedness. The proof is then closed by showing that the mean of the resulting stochastic solution, $\E[u]$, is consistent with our initial choice, that is, $\E[u] = m$.

\subsection{The mean field} \label{subsec:mean_field}

This section establishes key properties of the parabolic problem~\eqref{eq:mean_field_2}. While existence, uniqueness, and the maximum principle are standard results (see e.g.~\cite[Section~2.2]{malek_necas_1996} or~\cite[Appendix~B]{holden_risebro_2015}), we provide a detailed proof for the Hölder regularity of solutions.

\begin{definition} \label{def:mean_field_weak}
    Let $u_\initial \in L^\infty(\R^d)$. A \emph{weak solution} of equation~\eqref{eq:mean_field_2} is a function ${m \in L^\infty(\R^d \times (0, T))}$ which satisfies
    \begin{equation} \label{eq:mean_field_weak}
        \int_0^T \int_{\R^d} m \partial_t \varphi + f(m)\cdot \nabla \varphi + \frac{\eps^2}{2} m \Delta \varphi\, dx\,dt + \int_{\R^d} u_\initial(x) \varphi(x, 0)\, dx = 0
    \end{equation}
    for all $\varphi \in C^\infty_c(\R^d \times[0, T))$.
\end{definition}

Let $K_t^\eps$ denote the $d$-dimensional heat kernel with parameter $\eps > 0$, given by $K_t^\eps(x) \coloneqq (2 \pi \eps^2 t)^{-\nicefrac{d}{2}}\, \exp(-|x|^2/2 \eps^2 t)$. The following implicit solution formula is well-known.

\begin{lemma} \label{lemma:duhamel}
    Let $m$ be a weak solution of~\eqref{eq:mean_field_2}. Then it satisfies
    \begin{equation} \label{eq:duhamel}
        m(x, t) = (K_t^\eps * u_\initial)(x) - \int_0^t \bigl(\nabla K_{t-s}^\eps * f(m(s))\bigr)(x)\, ds
    \end{equation}
    for a.e.~$(x, t) \in \R^d \times [0, T]$.
\end{lemma}

This identity can be established using the weak formulation of the equation: For arbitrary smooth $\psi$, one constructs an auxiliary test function
\begin{equation*}
    \varphi(x, t) \coloneqq -\int_t^T \bigl(K_{s-t}^\eps(\cdot) * \psi(\cdot, s)\bigr)(x)\, ds,
\end{equation*}
which is a solution of the backward heat equation $\partial_t \varphi + \frac{\eps^2}{2}\Delta \varphi = \psi$  with terminal condition $\varphi(x, T) \equiv 0$. Applying Fubini's theorem to the resulting expression then yields the weak formulation of~\eqref{eq:duhamel}.

\begin{theorem} \label{thm:mean_field}
    Let $u_\initial \in L^\infty(\R^d)$ and $f \in C^1(\R; \R^d)$. Then there exists a unique weak solution of~\eqref{eq:mean_field_2}. It satisfies
    \begin{enumerate}[label=(\roman*)]
        \item maximum principle:
        \begin{equation*}
            \norm{m}_{L^{\infty}(\R^d \times (0, T))} \leq \norm{u_\initial}_{L^\infty(\R^d)} \qquad \text{for a.e.~} t > 0,
        \end{equation*}
        \item Hölder regularity:
        \begin{equation*}
            m \in L^q((0, T); C^{0, \beta}(\R^d)) \qquad \text{for } q \in [1, \infty),\ \beta\in\bigl(0,\min(1,\nicefrac2{q})\bigr)
        \end{equation*}
    \end{enumerate}
\end{theorem}

\begin{proof}
    The first part of the theorem is a standard result, see for instance~\cite[Theorem 4.48]{malek_necas_1996}. For the second part, we will use that the solution is given by the implicit formula~\eqref{eq:duhamel}. As a result of the smoothing property of the heat kernel, the solution $m$ is a continuous function from $(0, T)$ into $C^{0, \beta}(\R^d)$. Applying a difference operator $\Delta_h$ with $h \in \R^d$ to the solution formula yields
    \begin{equation*}
        \Delta_h[m(\cdot, t)] = \bigl(\Delta_h[K_t^\eps] * u_\initial\bigr) - \int_0^t \bigl(\Delta_h[\nabla K_{t-s}^\eps] * f(m(s))\bigr)\, ds.
    \end{equation*}
    Using Young's convolution inequality and Lemma~\ref{lemma:heat_estimates} in Appendix~\ref{app:heat-kernel-estimates} gives
    \begin{equation*}
    \begin{aligned}
        \bignorm{\Delta_h[m(\cdot, t)]}_{L^\infty(\R^d)} & \leq C(\beta, d, \eps) \frac{|h|^\beta}{t^{\frac{\beta}{2}}} \norm{u_\initial}_{L^\infty(\R^d)} + |f|_{\Lip(I)} |h|^\beta \int_0^t \frac{\norm{m(s)}_{L^\infty(\R^d)}}{(t-s)^\frac{1+\beta}{2}}\, ds \\
        & \leq C(\beta, d, \eps) \norm{u_\initial}_{L^\infty(\R^d)} |h|^\beta \biggl(\frac{1}{t^{\frac{\beta}{2}}} + |f|_{\Lip(I)} t^\frac{1-\beta}{2}\biggr)
    \end{aligned}
    \end{equation*}
    for a constant $C$ depending on $\beta$, $d$ and $\eps$, where $|f|_{\Lip(I)}$ denotes the Lipschitz constant of $f$ over the compact domain $I\coloneqq\bigl[\inf m, \sup m\bigr]$. Dividing by $|h|^\beta$ and taking the supremum over $|h| > 0$, we obtain the $C^{0, \beta}$ seminorm on the left-hand side. The remaining right-hand side is integrable in $L^q(0, T)$ for all $\beta < 2/q$.
\end{proof}

\subsection{Existence and uniqueness of solutions} \label{subsec:existence_uniqueness}

The goal of this section is to prove well-posedness of weak solutions for the stochastic mean-field conservation law~\eqref{eq:mean_field_cl_2}.

\begin{lemma}[The mean-field equation] \label{lemma:mean_field_eq}
    Let $u$ be a weak solution of~\eqref{eq:mean_field_cl_2}. Then its mean, $m := \E[u]$, is the unique weak solution of the viscous conservation law~\eqref{eq:mean_field_2} (whose well-posedness is guaranteed by Theorem~\ref{thm:mean_field}).
\end{lemma}

\begin{proof}
	For every \(\vartheta\in C_c^\infty(\mathbb R^d)\), the additional local integrability assumption~\eqref{eq:cl-weak-soln-integrability} gives
    \begin{equation*}
    \E\int_0^T \left| \int_{\mathbb R^d} u_r(x)\nabla\vartheta(x)\,dx \right|^2 dr < \infty.
    \end{equation*}
    Hence, the It\^o integral in the weak formulation~\eqref{eq:cl_mean_field_weak} is a square-integrable martingale and has zero expectation. Since \(m\) and therefore \(a(m)\) are bounded, the remaining terms are integrable, and Fubini's theorem applies. Thus, taking expectation of~\eqref{eq:cl_mean_field_weak} in Itô form (as in~\eqref{eq:mean_field_cl_weak_ito} with $b\coloneqq a(m)$ and noise coefficient $\varepsilon$), we see that $m$ satisfies
    \begin{equation} \label{eq:mean_field_linear}
        \int_{\R^d} \vartheta m(t)\, dx = \int_{\R^d} \vartheta u_\initial\, dx + \int_0^t \int_{\R^d} f(m(r)) \cdot \nabla \vartheta + \frac{\eps^2}{2} m(r) \Delta \vartheta\, dx\,dr
    \end{equation}
    for all $t \in [0, T]$ and $\vartheta \in C^\infty_c(\R^d)$. The spacetime weak formulation~\eqref{eq:mean_field_weak} and the evolutionary formulation~\eqref{eq:mean_field_linear} are equivalent, as can be shown by a standard argument (see e.g.~\cite[Proposition 6.1.2]{bogachev_krylov_rockner2015}).
\end{proof}

The previous lemma shows that the mean of any weak solution of~\eqref{eq:mean_field_cl_2} is a weak solution of the viscous conservation law~\eqref{eq:mean_field_2}. We now establish the converse: that any solution $m$ of \eqref{eq:mean_field_2} admits a stochastic representation, in the sense that there exists a weak solution $u$ of \eqref{eq:mean_field_cl_2} whose mean is $m$. The proof is constructive: We begin by fixing a deterministic solution $\bar{m}$, and use it as a coefficient in an auxiliary linear stochastic continuity equation. We then show that the mean of the resulting process must, by uniqueness, be identical to $\bar{m}$, thereby verifying consistency of the representation.

\begin{proof}[Proof of Theorem~\ref{thm:main_spde}]
    Given $u_\initial \in L^\infty(\R^d)$, we begin by fixing a function $\bar{m}$, defined as the unique weak solution of the viscous conservation law~\eqref{eq:mean_field_2} with initial data $u_\initial$, and consider the associated linear stochastic continuity equation
    \begin{equation} \label{eq:auxiliary_stoch_cont}
        du + \nabla \cdot \bigl(a(\bar{m}) u\bigr)\, dt + \eps \nabla u \circ dW_t = 0,
    \end{equation}
    also with initial data $u_\initial$. Since we have assumed that $f \in C^{1, \alpha}_\loc(\R; \R^d)$, the function $a(v) \coloneqq \frac{f(v)-f(0)}{v}$ is at least $C^{0, \alpha}_\loc(\R; \R^d)$. Combined with $\bar{m}$ being bounded and in $L^2((0, T); C^{0, \beta}(\R^d))$ for any $\beta \in (0, 1)$, the composition $a\circ \bar{m}$ is bounded with $\alpha \beta$-Hölder seminorm
    \begin{equation*}
        \bigl[a(\bar{m}(\cdot, t))\bigr]_{C^{0, \alpha \beta}(\R^d; \R^d)} \leq [a]_{C^{0, \alpha}(I; \R^d)} [\bar{m}(\cdot, t)]_{C^{0, \beta}(\R^d)}^\alpha
    \end{equation*}
    for a.e.~$t \in (0, T)$, where the H\"older seminorm of $a$ is taken over the compact domain $I\coloneqq\bigl[\inf\bar m, \sup\bar m\bigr]$. This implies in particular that $a \circ \bar{m} \in L^2((0, T); C^{0, \alpha \beta}(\R^d; \R^d))$. Consequently, the velocity field $a(\bar{m})$ in the auxiliary equation~\eqref{eq:auxiliary_stoch_cont} satisfies the assumptions of Theorem~\ref{thm:linear_continuity}, meaning that there is a unique weak solution $u$ of this equation, given by the formula
    \begin{equation*}
        \begin{aligned}
            \int_{\R^d} \vartheta(x) u_t(x)\, dx & = \int_{\R^d} \vartheta(X_t(x)) u_\initial(x)\, dx \\
            & = \int_{\R^d} \vartheta(x) u_\initial\bigl(X_t^{-1}(x)\bigr) \det\bigl(\nabla X_t^{-1}(x) \bigr)\, dx
        \end{aligned}
    \end{equation*}
    for all $t \in [0, T]$ and $\vartheta \in C_c(\R^d)$, $\Pas$, where $X$ is a stochastic flow of \diffeomorphisms generated by~\eqref{eq:stochastic_flow} with drift $a(\bar{m})$. Fix $K\subset \R^d$, with $K \subset B_R(0)$, and set
    \begin{equation*}
    M \coloneqq \norm{u_\initial}_{L^\infty(\R^d)}, \qquad
    A_M \coloneqq \max_{|v|\leq M}|a(v)|.
    \end{equation*}
    The SDE for the flow gives
    \begin{equation*}
    X_t^{-1}(K) \subset B_{R+tA_M+\varepsilon|W_t|}(0).
    \end{equation*}
    Consequently,
    \begin{equation*}
    \norm{u_t}_{L^1(K)} \leq |B_1(0)| \norm{u_\initial}_{L^\infty(\R^d)}(R+tA_M+\varepsilon|W_t|)^d.
    \end{equation*}
    Since the Brownian motion has finite moments of every order, we conclude that the integrability condition~\eqref{eq:cl-weak-soln-integrability} holds.

    Proceeding as in Lemma~\ref{lemma:mean_field_eq}, converting equation~\eqref{eq:auxiliary_stoch_cont} to Itô form introduces a second-order correction term. Taking the expectation, the Itô martingale term vanishes, and we find that $m \coloneqq \E[u]$ satisfies, in the weak sense, the linear parabolic PDE
    \begin{equation*}
        \partial_t m + \nabla \cdot \bigl(a(\bar{m}) m\bigr) = \tfrac{\eps^2}{2} \Delta m
    \end{equation*}
    with initial condition $u_\initial$. By construction, $\bar{m}$ is also a solution of this linear PDE, since $\nabla \cdot (a(\bar{m}) \bar{m}) = \nabla \cdot f(\bar{m})$, and therefore, by uniqueness of weak solutions, we must have $m = \bar{m}$. This verifies the stochastic representation of $\bar{m}$, showing that $\bar{m} = m = \E[u]$.

    To prove uniqueness, suppose that there are two weak solutions, $u$ and $v$, of~\eqref{eq:mean_field_cl}. Then by Lemma~\ref{lemma:mean_field_eq}, we must have $\E[u] = \E[v] = m$, where $m$ is the unique weak solution of \eqref{eq:mean_field_2}. Consequently, both $u$ and $v$ solve the same linear stochastic continuity equation with velocity field $a(m)$, whose uniqueness is guaranteed by Theorem~\ref{thm:linear_continuity}. This implies that $u = v$ (almost everywhere on $\R^d \times (0, T)$, $\Pas$), completing the proof.
\end{proof}

\section{Filippov's differential inclusion and zero-noise limit of SDEs} \label{sec:filippov}

This section establishes a convergence result for the zero-noise limit of the SDE
\begin{equation} \label{eq:general_sde}
    dX_t^\eps = b^\eps(X_t^\eps, t)\, dt + \eps\, dW_t.
\end{equation}
We assume that the drift $b^\eps$ is bounded, measurable, and satisfies certain weak convergence criteria (including cases like one-sided Lipschitz functions, which are relevant to scalar conservation laws). Under these conditions, we prove that the solutions $X^\eps$ converge in law to some limit process $X$ that is almost surely a Filippov solution of a limiting ODE
\begin{equation} \label{eq:limit_ode}
    \tfrac{d}{dt}{X}_t = b(X_t, t).
\end{equation}
Our framework accommodates time-dependent drifts and relaxes the uniform convergence assumption used in related works such as~\cite{buckdahn_ouknine_2009}. The analysis is self-contained and provides the technical foundation needed for the remaining part of the paper.

\subsection{Filippov solutions} \label{subsec:filippov_solutions}

The Filippov solution concept is a way to define solutions for ODEs where the right-hand side is irregular (e.g.~discontinuous). The core idea is to replace the right-hand side by a set-valued map that captures its essential limiting behavior in the neighborhood of any point.

Let $\overline{\conv}(A)$ denote the smallest closed convex set containing $A \subset \R^d$. For a bounded and measurable function $b\colon \R^d \times (0, T) \to \R^d$, define the set-valued map
\begin{equation*}
    K_R[b](x, t) \coloneqq \bigcap_{\substack{N \subset \R^d \\ |N| = 0}}\overline{\conv} \bigl(b(B_{R}(x) \setminus N, t)\bigr),
\end{equation*}
i.e.~the closed convex hull of essential range of $b$ over the ball $B_R(x)$ of radius $R > 0$ centered at $x \in \R^d$. Let furthermore $K[b](x, t)\coloneqq \bigcap_{R>0} K_R[b](x, t)$. A \emph{Filippov solution} of the ODE~\eqref{eq:limit_ode} is an absolutely continuous function $t \mapsto X_t$ that satisfies the differential inclusion
\begin{equation*}
    \tfrac{d}{dt}{X}_t \in K[b](X_t, t) \quad \text{for a.e. }t \in (0, T).
\end{equation*}
As shown by Filippov~\cite{Filippov60}, this is equivalent to the condition that for every direction $v \in \R^d$, one has
\begin{equation*}
    \frac{d{X}_t}{dt} \cdot v \leq M[b \cdot v](X_t, t) \quad \text{for a.e. }t \in (0, T),
\end{equation*}
where $M$ denotes the limiting local essential supremum, defined by
\begin{equation*}
    M[c](x) \coloneqq \lim_{R \to 0} M_R[c](x), \qquad M_R[c](x) \coloneqq \esssup_{y \in B_R(x)} c(y)
\end{equation*}
for any $c\in L^\infty(\R^d)$.

\subsection{Zero-noise limit of SDEs}

The noise in the SDE~\eqref{eq:general_sde} has a regularizing effect, ensuring the equation is well-posed even for irregular drifts $b^\eps$. Indeed, as mentioned in the introduction, a classical result by Zvonkin~\cite{zvonkin_1974} and Veretennikov~\cite{veretennikov_1981} guarantees that a unique strong solution exists if the drift is merely bounded and measurable. The boundedness of the drift allows for standard moment estimates that establish tightness of the sequence of solutions $(X^\eps)_{\eps > 0}$, which, by Prokhorov's theorem, implies existence of a subsequence converging in law to some process $X$. The characterization of this limiting process depends on the mode of convergence of the drift $b^\eps$.

For some radius $R > 0$, we define a time-dependent distance between the vector fields $b^\eps$ and $b$ by
\begin{equation*}
    d_R(b^\eps; b)(t) \coloneqq \esssup_{x \in \R^d} \dist \bigl(b^\eps(x, t), K_R[b](x, t)\bigr).
\end{equation*}
This is a well-defined, measurable function of time $t$. We now show that convergence with respect to the distance $d$ is a sufficient condition for weak convergence of solutions of the SDE to a limiting Filippov solution on the space of continuous paths, $\C \coloneqq C([0, T]; \R^d)$, equipped with the supremum norm.

\begin{theorem} \label{thm:filippov_convergence}
    Let $b^\eps, b: \mathbb{R}^d \times (0, T) \to \mathbb{R}^d$ be bounded and measurable vector fields. For $\eps > 0$, let $X^\eps$ be the solution of the SDE
    \begin{equation} \label{eq:general_sde_integral}
        X_t^\eps = x + \int_0^t b^\eps(X_r^\eps, r)\, dr + \eps W_t.
    \end{equation}
    Assume that for all $R > 0$, the vector fields $b^\eps$ converge in the sense that
    \begin{equation}\label{eq:convergence-in-local-convex-hull}
        \lim_{\eps \to 0} \bignorm{d_R(b^\eps; b)}_{L^1(0, T)} = 0.
    \end{equation}
    Then the sequence of processes $(X^\eps)_{\eps > 0}$ is tight in $\C$, and any of its limit points is $\Pas$ a Filippov solution of the limiting ODE~\eqref{eq:limit_ode}.
\end{theorem}

\begin{remark}
    The convergence $b^\eps \to b$ is measured in the distance $d_R$, since $L^\infty$-convergence of the drift coefficients is an overly strong condition in the presence of discontinuities. This framework closely follows Filippov's classical stability theorem, which ensures convergence of solutions when the time integral of the distance between the drift coefficients and the limiting Filippov set-valued map vanishes \cite{Filippov60}.
\end{remark}

\begin{proof}
    For all $\eps > 0$ and any $x \in \R^d$, there exists a unique strong solution of~\eqref{eq:general_sde}, i.e.~a stochastic process $X^\eps$ on the probability space $(\Omega, \F, \P)$, adapted to the filtration $(\F_t)$, which satisfies~\eqref{eq:general_sde_integral} for all $t \in [0, T]$, $\Pas$. Boundedness of the drift ensures tightness of the sequence of solutions $(X^\eps)$, which in turn implies tightness of the sequence of joint laws for $(X^\eps, W)$. By Prokhorov's theorem, there exists a subsequence ${\eps_k \to 0}$ such that $(X^{\eps_k}, W)_\# \P$ converges weakly to some probability measure $\mu$ on $\C \times \C$. Using Skorokhod's theorem, it is possible to construct a new probability space $(\tilde{\Omega}, \tilde{\F}, \tilde{\P})$ and new stochastic processes $\tilde{X}^{\eps_k}, \tilde{W}^{\eps_k}$ and $\tilde{X}, \tilde{W}$ such that
    \begin{equation} \label{eq:as_convergence}
        \tilde{X}^{\eps_k} \to \tilde{X}, \quad \tilde{W}^{\eps_k} \to \tilde{W} \qquad \text{in } \C, \quad  \tilde{\P}\text{-a.s.},
    \end{equation}
    and
    \begin{equation*}
        (\tilde{X}^{\eps_k}, \tilde{W}^{\eps_k})_{\#} \tilde{\P} = (X^{\eps_k}, W)_{\#} \P, \qquad (\tilde{X}, \tilde{W})_{\#} \tilde{\P} = \mu.
    \end{equation*}
    Setting $\tilde{Y}^{\eps_k} \coloneqq \tilde{X}^{\eps_k} - \eps_k \tilde{W}^{\eps_k}$, we have
    \begin{equation} \label{eq:y}
        \tilde{Y}^{\eps_k}_t = x + \int_0^t b^{\eps_k}(\tilde{X}^{\eps_k}_r, r)\, dr \qquad \tilde{\P}\text{-a.s.},
    \end{equation}
    and $\tilde{Y}^{\eps_k} \to \tilde{X}$ in $\C$ in view of~\eqref{eq:as_convergence} (see e.g.~\cite[Chapter 2.6]{krylov_1980} for an argument that the new processes $\tilde{X}^{\eps_k}, \tilde{W}^{\eps_k}$ satisfy an analogous SDE on $(\tilde{\Omega}, \tilde{\F}, \tilde{\P})$). Fix a radius $R > 0$. The drift $b^{\eps_k}(\tilde{X}^{\eps_k})$ in~\eqref{eq:y} can be decomposed into one part which belongs to $K_{R/2}[b](\tilde{X}^{\eps_k})$ and one part which is the corresponding difference, i.e.
    \begin{equation} \label{eq:drift_decomposition}
        b^{\eps_k}(\tilde{X}^{\eps_k}) = y + \bigl(b^{\eps_k}(\tilde{X}^{\eps_k}) - y\bigr),
    \end{equation}
    where $y \in K_{R/2}[b](\tilde{X}^{\eps_k})$ is the point which minimizes $|b^{\eps_k}(\tilde{X}^{\eps_k})-y|$. Note that for $\tilde{\P}$-a.e.~$\omega \in \tilde{\Omega}$, there exists $N \in \N$ such that ${\norm{\tilde{X}^{\eps_k}(\omega) - \tilde{X}(\omega)}_{\C} < R/2}$ for all $k \geq N$. Fix such $\omega$ and $N$ in the following few calculations. Taking the inner product between~\eqref{eq:drift_decomposition} and an arbitrary $v \in \R^d$, this yields
    \begin{equation} \label{eq:distance_inequality}
    \begin{aligned}
        b^{\eps_k}(\tilde{X}_t^{\eps_k}, t) \cdot v & \leq M_{R/2}[b\cdot v](\tilde{X}_t^{\eps_k}, t) + d_{R/2}(b^{\eps_k}; b)(t) |v| \\
        & \leq M_R[b\cdot v](\tilde{X}_t, t) + d_{R/2}(b^{\eps_k}; b)(t) |v|
    \end{aligned}
    \end{equation}
    for a.e.~$t \in (0, T)$, for all $k \geq N$. Using this inequality in~\eqref{eq:y}, we obtain
    \begin{equation*}
    \begin{aligned}
        \big(\tilde{Y}^{\eps_k}_t - \tilde{Y}^{\eps_k}_s\big) \cdot v & = \int_s^t b^{\eps_k}(\tilde{X}_r^{\eps_k}, r) \cdot v\, dr \\
        & \leq \int_s^t M_R[b\cdot v](\tilde{X}_r, r) + d_{R/2}(b^{\eps_k}; b)(r) |v|\, dr.
    \end{aligned}
    \end{equation*}
    Although the inequality~\eqref{eq:distance_inequality} holds only for almost every point in $\R^d$, this is sufficient for our pathwise analysis since the law of the process $\tilde{X}_t^{\eps_k}$ is absolutely continuous with respect to the Lebesgue measure for all $t > 0$. Consequently, the probability of the process occupying the null set where the inequality might fail is zero, ensuring that the bound holds $\tilde{\P}$-a.s. Passing $k \to \infty$ and using the convergence~\eqref{eq:convergence-in-local-convex-hull} of $d_R$, we see that the limit $\tilde{X}$ is a Lipschitz function which satisfies
    \begin{equation*}
        (\tilde{X}_t - \tilde{X}_s) \cdot v \leq \int_s^t M_R[b\cdot v](\tilde{X}_r, r)\, dr
    \end{equation*}
     for all $0 < s < t < T$, $\tilde{\P}$-a.s. Since this is valid for arbitrary $R> 0$, the Lipschitz continuous process $t \mapsto \tilde{X}$ must be a Filippov solution of~\eqref{eq:limit_ode}, $\tilde{\P}$-a.s. Finally, let $F$ be the set of all Filippov solutions of~\eqref{eq:limit_ode}. Since $\tilde{X}$ and $X$ have the same law, we have $\tilde{\P}\bigl(\tilde{X} \in F\bigr) = \P\bigl(X \in F\bigr) = 1$, meaning that $X$ is $\P$-a.s.~a Filippov solution of~\eqref{eq:limit_ode}. The same argument applies to any weakly convergent subsequence.
\end{proof}

\section{Zero-noise limit of the SPDE}\label{sec:zero-noise}

This section is dedicated to proving Theorem~\ref{thm:main_zero_noise}, which establishes convergence of the solution $u^\eps$ of the stochastic mean-field equation~\eqref{eq:mean_field_cl} to the entropy solution of the conservation law~\eqref{eq:cl}. Our analysis is performed in one dimension, under the assumption that the flux $f$ is uniformly convex, and that the initial data $u_\initial$ belongs to $\BV_\loc \cap L^\infty \cap L^1(\R)$.

Building on the stochastic representation of the viscous conservation law~\eqref{eq:mean_field} established in Theorem~\ref{thm:main_spde}, our strategy is to show that the associated stochastic flow $X^\eps$ converges. Specifically, we will prove that this flow, governed by the SDE
\begin{equation}\label{eq:stochastic_flow_2}
    \begin{cases}
        dX_t^\eps = a(m^\eps(X_t^\eps,t))\,dt + \eps\,dW_t &\text{for}\ s < t < T, \\
        X_s^\eps = x,
    \end{cases}
\end{equation}
converges, for each $x \in \R$, as $\eps \to 0$, to the unique Filippov flow
\begin{equation} \label{eq:cl_flow_2}
    \begin{cases}
        \tfrac{d}{dt}{X}_t = a(u(X_t, t)) &\text{for}\ s < t < T, \\
        X_s = x,
    \end{cases}
\end{equation}
of the limiting conservation law \eqref{eq:cl}. Recall that the well-posedness of the Filippov flow $X$ of \eqref{eq:cl_flow_2} was established in~\cite{fjordholm_maehlen_oerke} (see Section~\ref{subsec:particle_paths}).

We first establish a technical lemma which makes our results from Section~\ref{sec:filippov} applicable to the current problem. To this end, let the one-sided Lipschitz constant $|b|_{\Lip^+}$ of a function $b$ be defined as
\begin{equation*}
    |b|_{\Lip^+} \coloneqq \esssup_{x\neq y}\frac{b(x)-b(y)}{x-y}.
\end{equation*}

\begin{lemma} \label{lemma:drift_convergence}
    Let $b^k, b$ be functions in $L^\infty(\R \times (0, T))$ for all $k\in\N$. Assume that
    \begin{enumerate}[label=(\roman*)]
        \item there exists a function $0\leq L\in L^{\nicefrac{1}{2}}([0,T])$ such that $|b^k(\cdot, t)|_{\Lip^+} \leq L(t)$ for a.e.~$t \in (0, T)$ and all $k\in\N$,
        \item the following limit holds:
        \begin{equation} \label{eq:l1_lip_limit}
            \lim_{k \to \infty} \norm{b^k - b}_{L^\infty((0,T);L^1(\R))} = 0.
        \end{equation}
    \end{enumerate}
    Then, for any radius $R > 0$, we have
    \begin{equation} \label{eq:dist_limit}
        \lim_{k \to \infty} \bignorm{d_R(b^k; b)}_{L^1(0, T)} = 0.
    \end{equation}
\end{lemma}

\begin{proof}
    Start by fixing $k \in \N$, $R > 0$, $\delta > 0$ and $t \in (0, T)$. Choose a $\delta$-optimal point~$x_0 \in \R$ for the distance $d_R(b^k, b)(t)$, i.e.
    \begin{equation*}
        d_R(b^k; b)(t) < \dist \bigl(b^k(x_0, t), K_R[b](x_0, t)\bigr) + \delta.
    \end{equation*}
    For any $0 < r \leq R$, we have
    \begin{equation} \label{eq:min_identity}
    \begin{aligned}
        \dist \bigl(b^k(x_0, t), K_R[b](x_0, t)\bigr) & = \min_{y \in K_R[b](x_0, t)} |b^k(x_0, t) - y| \\
        & \leq \frac{1}{r} \int_{x_0 - r}^{x_0} |b^k(x_0, t) - b(x, t)|\, dx.
    \end{aligned}
    \end{equation}
    Assume without loss of generality that $b^k(x_0, t) > b(x, t)$ for a.e.~$x \in B_R(x_0)$ (since if $b^k(x_0, t) - b(x, t)$ were to change sign, then $b^k(x_0, t) \in K_R[b](x_0, t)$ and we would get $d_R(b^k;b)(t) <\delta$, where $\delta$ is arbitrarily small). Then the right-hand side of~\eqref{eq:min_identity} can be estimated by splitting the integral and using the one-sided Lipschitz condition:
    \begin{equation*}
    \begin{aligned}
        & \frac{1}{r} \int_{x_0 - r}^{x_0} |b^k(x_0, t) - b(x, t)|\, dx \\
        & = \frac{1}{r} \int_{x_0 - r}^{x_0} b^k(x_0, t) - b^k(x, t)\, dx + \frac{1}{r} \int_{x_0 - r}^{x_0} b^k(x, t) - b(x, t)\, dx \\
        & \leq \frac{|b^k(\cdot, t)|_{\Lip^+}}{r} \int_{x_0 - r}^{x_0} (x_0 - x)\, dx + \frac{1}{r} \int_{x_0 - r}^{x_0} |b^k(x, t) - b(x, t)|\, dx \\
        & \leq \frac{r}{2} L(t) + \frac{1}{r} \norm{b^k - b}_{L^\infty((0,T);L^1(\R))}.
    \end{aligned}
    \end{equation*}
	Assume without loss of generality that $L(t)\geq1$ for all $t$, and that $k\in\N$ is large enough that $\|b^k-b\|_{L^\infty((0,T);L^1(\R))}\leq R^2/2$ (which is possible due to assumption (ii)). The choice $r \coloneqq \sqrt{2 \norm{b^k - b}_{L^\infty((0,T);L^1(\R))} / L(t)}$ will then satisfy $r\leq R$, and yields
    \begin{equation*}
        d_R(b^k; b)(t) \leq \sqrt{2L(t) \norm{b^k - b}_{L^\infty((0,T);L^1(\R))}} + \delta.
    \end{equation*}
    Since $\delta > 0$ was arbitrary, integrating over $(0, T)$ and taking the limit $k \to \infty$ proves~\eqref{eq:dist_limit}.
\end{proof}

Using the above lemma, we are able to prove the following key proposition.

\begin{proposition} \label{prop:flow_convergence_pointwise}
    Under the same assumptions as in Theorem~\ref{thm:main_zero_noise}, and
    for any initial condition $x \in \R$, let $X^\eps(x)$ and $X(x)$ be the unique strong solution of~\eqref{eq:stochastic_flow_2} and the unique Filippov solution of~\eqref{eq:cl_flow_2}, respectively, with $X_0^\eps = X_0 = x$. Then $X^\eps(x)$ converges to $X(x)$ in $L^p(\Omega)$ in the uniform topology: for any $p\geq1$,
    \begin{equation} \label{eq:path_convergence}
        \lim_{\eps \to 0} \E\biggl[ \sup_{t \in [0, T]} |X_t^\eps(x) - X_t(x)|^p \biggr] = 0.
    \end{equation}
\end{proposition}

\begin{proof}
    We first show that $X^\eps$ converges to $X$ in probability. It is a standard result that when $u_\initial \in \BV_\loc \cap L^\infty \cap L^1(\R)$, then the viscous solution $m^\eps$ of~\eqref{eq:mean_field} and the entropy solution $u$ of~\eqref{eq:cl} satisfy
    \begin{equation*}
        m^\eps, u \in L^\infty(\R \times (0, T)) \cap C([0, T]; L^1(\R)), \qquad m^\eps \to u\ \text{in}\ C([0, T]; L^1(\R)) \text{ as } \eps\to0,
    \end{equation*}
    and moreover that for all $t > 0$, the Oleinik estimates
    \begin{equation*}
        |m^\eps(\cdot, t)|_{\Lip^+} \leq \frac{1}{t \inf(f'')}, \qquad |u(\cdot, t)|_{\Lip^+} \leq \frac{1}{t \inf(f'')}
    \end{equation*}
    hold, where $\inf(f'') > 0$, by assumption (see~\cite{hoff_1983, oleinik_1963}). Since we have assumed that $f \in C^2$, we get $a \in C^1$ (recall that $a(u) \coloneqq (f(u)-f(0))/u$), and along with the uniform $L^\infty$ bound on $(m^\eps)_{\eps>0}$, the assumptions of Lemma~\ref{lemma:drift_convergence} are fulfilled for the sequence of drifts $a(m^\eps)$ converging to $a(u)$.
    Lemma~\ref{lemma:drift_convergence} and Theorem~\ref{thm:filippov_convergence} then imply that the sequence of solutions $(X^\eps)_{\eps>0}$ is tight in $\C$ (here, $\C \coloneqq C([0, T])$). Since any convergent subsequence has the same limit point---the unique Filippov solution $X$ of~\eqref{eq:cl_flow_2}---the entire sequence converges in distribution to $X$. Moreover, because the limit $X$ is deterministic, this is equivalent to convergence in probability:
    \begin{equation*}
        \lim_{\eps \to 0} \P\bigl( \norm{X^\eps - X}_{\C} > \delta \bigr) = 0 \quad \text{for any } \delta > 0.
    \end{equation*}
    Next, we upgrade this result to convergence in expectation. The coefficients of the SDEs are uniformly bounded by assumption (the diffusion is even constant).
	Fix $p \geq 1$, and choose $q>p$. Standard moment estimates for SDEs with uniformly bounded drift yield
    \begin{equation*}
    \sup_{0<\varepsilon\le 1}
    \E\bigl[\|X^\varepsilon\|_{\C}^{q}\bigr]
    <\infty.
    \end{equation*}
    Since $X$ is deterministic, it follows that
    \begin{equation*}
    \sup_{0<\varepsilon\le 1}
    \E \bigl[\|X^\varepsilon-X\|_{\C}^{q}\bigr]
    <\infty.
    \end{equation*}
    Hence, the family
    \begin{equation*}
    \bigl(
    \|X^\varepsilon-X\|_{\C}^{p}
    \bigr)_{0<\varepsilon\le 1}
    \end{equation*}
    is uniformly integrable. Therefore, by Vitali's convergence theorem,
    \begin{equation*}
    \lim_{\varepsilon\to 0}
    \E \bigl[
    \|X^\varepsilon-X\|_{\C}^{p}
    \bigr]
    =0.    \qedhere
    \end{equation*}
\end{proof}

We are now in position to prove the main theorem about the zero-noise limit of the SPDE~\eqref{eq:mean_field_cl}.

\begin{proof}[Proof of Theorem~\ref{thm:main_zero_noise}]
    We aim to prove the convergence
        \begin{equation*}
        \lim_{\eps \to 0} \E\biggl[\sup_{t \in [0, T]} \biggl|\int_{\R} \vartheta(x) \bigl(u_t^\eps(x) - u(x, t)\bigr)\, dx\biggr|^p \biggr] = 0
    \end{equation*}
    for any $\vartheta \in C_b(\R)$, where $u$ is the unique entropy solution of \eqref{eq:cl}. We first prove the claim for $p = 1$. Recall that for each $\eps > 0$, the weak solution $u^\eps$ of~\eqref{eq:mean_field_cl} is given by the pushforward formula $u^\eps = X^\eps_{\#} u_\initial$, where $X^\eps$ is the stochastic flow of \diffeomorphisms generated by~\eqref{eq:stochastic_flow_2}. Moreover, by Theorem~\ref{thm:particle_paths}, the entropy solution $u$ of \eqref{eq:cl} is given by $u = X_\# u_\initial$, where $X$ is the unique Filippov flow generated by \eqref{eq:cl_flow}. Changing variables according to the pushforward formula and applying the triangle inequality and Fubini's theorem, we obtain the estimate
    \begin{equation} \label{eq:weak_convergence_estimate}
        \begin{aligned}
            & \E\biggl[\sup_{t \in [0, T]} \biggl|\int_{\R} \vartheta(x) \bigl(u_t^\eps(x) - u(x, t)\bigr)\, dx\biggr|\biggr] \\
            & \qquad\leq \int_\R \E\biggl[\sup_{t \in [0, T]} \bigl|\vartheta(X_t^\eps(x)) - \vartheta(X_t(x))\bigr|\biggr] |u_\initial(x)|\, dx.
        \end{aligned}
    \end{equation}
    We will show that the integral over $\R$ on the right-hand side converges to zero by dominated convergence. To this end, Proposition~\ref{prop:flow_convergence_pointwise} implies that, for any $x \in \R$, $X^\eps(x)$ converges to $X(x)$ in probability, and thus, by the Continuous Mapping Theorem (see e.g.~\cite[Lemma~5.3]{kallenberg_2021}),
    \begin{equation*}
        \vartheta(X^\eps(x)) \xrightarrow{\ p\ } \vartheta(X(x)) \qquad \text{in } \C
    \end{equation*}
    for any $\vartheta \in C_b(\R)$. Since $\vartheta$ is bounded, the family $\norm{\vartheta(X^\eps(x)) - \vartheta(X(x))}_\C$ is uniformly integrable, and hence the Vitali Convergence Theorem implies that, for every $x \in \R$,
    \begin{equation*}
        \lim_{\eps \to 0} \E\biggl[\sup_{t \in [0, T]}\bigl|\vartheta(X_t^\eps(x)) - \vartheta(X_t(x))\bigr|\biggr] = 0.
    \end{equation*}
    Finally, since $|\vartheta(X_t^\eps(x)) - \vartheta(X_t(x))| \leq 2 \norm{\vartheta}_{L^\infty}$, the integrand in \eqref{eq:weak_convergence_estimate} is dominated by the integrable function $2 \norm{\vartheta}_{L^\infty} |u_\initial|$. The dominated convergence theorem applied to~\eqref{eq:weak_convergence_estimate} therefore proves the claim for $p=1$. To obtain the general case, set
    \begin{equation*}
    Z^\eps \coloneqq
    \sup_{t\in[0,T]}
    \left|
    \int_{\R}\vartheta(x)\bigl(u_t^\eps(x)-u(x,t)\bigr)\,dx
    \right|.
    \end{equation*}
    The pushforward representations of $u^\eps$ and $u$ imply that, $\Pas$,
    \begin{equation*}
    0\leq Z^\eps \leq 2\norm{\vartheta}_{L^\infty(\R)}
    \norm{u_\initial}_{L^1(\R)}
    \eqqcolon C.
    \end{equation*}
    Hence, for every $p > 1$,
    \begin{equation*}
    \E\bigl[(Z^\eps)^p\bigr]
    \leq C^{p-1}\E[Z^\eps]
    \longrightarrow 0,
    \end{equation*}
    which concludes the proof.
\end{proof}

\section{Conclusions and future work}

This paper establishes well-posedness of a novel mean-field stochastic perturbation of scalar conservation laws and proves its convergence to the unique entropy solution in the zero-noise limit. The analysis of the zero-noise limit is restricted to one spatial dimension, and relies on specific structural assumptions. A primary direction for future research is the extension to multiple dimensions. This represents a significant challenge, as the entropy characterization via particle paths \eqref{eq:cl_flow}, which we heavily rely on here, does not readily generalize. Further research should also aim to relax the technical assumptions on the initial data, particularly the $L^1(\R)$ condition, and investigate the behavior of the zero-noise limit for non-convex flux functions.

\section*{Acknowledgements}
The first author gratefully acknowledges the support of the Research Council of Norway through the project \textit{INICE} (301538).

\appendix

\section{Heat kernel estimates}\label{app:heat-kernel-estimates}

The following elementary lemma is used to quantify the smoothing property of the heat kernel. Here, $\Delta_h$ denotes the difference operator $\Delta_h[u](x) = u(x+h) - u(x)$ for a vector $h \in \R^d$.

\begin{lemma} \label{lemma:heat_estimates}
    Let $\beta \in (0, 1]$. Then there exists a constant $C = C(\beta,d)$ such that
    \begin{equation} \label{eq:heat_kernel_differences}
        \bigl\|\Delta_{h}[K^\eps_t]\bigr\|_{L^1(\R^d)} \leq C\frac{|h|^\beta}{(\eps^2 t)^{\frac{\beta}{2}}} \quad \text{and}\quad \bigl\|\Delta_{h}[\partial_i K^\eps_t]\bigr\|_{L^1(\R^d)} \leq C\frac{|h|^\beta}{(\eps^2 t)^\frac{1+\beta}{2}}
    \end{equation}
    for all $t \in (0, T)$, $h \in \R^d$ and $i = 1, \dots, d$.
\end{lemma}

\begin{proof}
    Let $h \in \R^d$. Then
    \begin{equation*}
        \begin{aligned}
            \bignorm{\Delta_h[K_t^\eps]}_{L^1} & = \int_{\R^d} |K_t^\eps(x+h) - K_t^\eps(x)|\, dx = \int_{\R^d} \biggl| \int_0^1 \nabla K_t^\eps(x+rh)\cdot h\, dr \biggr|\, dx \\
            & \leq |h| \int_0^1 \int_{\R^d} \bigl|\nabla K_t^\eps(x + rh)\bigr|\, dx\,dr \lesssim_d \frac{|h|}{\eps \sqrt{t}}.
        \end{aligned}
    \end{equation*}
    If $|h|/(\eps \sqrt{t}) < 1$, then the desired inequality holds upon raising to a power $\beta$. If on the other hand ${|h|/(\eps \sqrt{t}) \geq 1}$, then
    \begin{equation*}
        \bignorm{\Delta_h[K_t^\eps]}_{L^1} \leq 2 \norm{K_t^\eps}_{L^1} =2 \lesssim \frac{|h|^\beta}{(\eps^2 t)^{\frac{\beta}{2}}},
    \end{equation*}
    and the inequality is still valid. Similarly, the difference operator applied to the partial derivative can be bounded by
    \begin{equation*}
        \bignorm{\Delta_h[\partial_i K_t^\eps]}_{L^1} \leq |h| \int_0^1 \int_{\R^d} \bigl|\partial_i \nabla K_t^\eps(x + rh)\bigr|\, dx\,dr \lesssim_{d} \frac{|h|}{\eps^2 t}.
    \end{equation*}
    Assuming that $|h|/(\eps\sqrt{t}) < 1$, we find
    \begin{equation*}
        \bignorm{\Delta_{h}[\partial_i K_t^\eps]}_{L^1} \lesssim \frac{1}{\eps \sqrt{t}} \biggl(\frac{|h|}{\eps \sqrt{t}}\biggr)^\beta = \frac{|h|^\beta}{(\eps^2 t)^{\frac{1+\beta}{2}}}.
    \end{equation*}
    On the other hand, if $\frac{|h|}{\eps\sqrt{t}} \geq 1$ we have
    \begin{equation*}
        \norm{\Delta_{h}[\partial_i K_t^\eps]}_{L^1} \leq 2 \norm{\partial_i K_t^\eps}_{L^1} \lesssim_{d} \frac{1}{\eps \sqrt{t}} \leq \frac{|h|^\beta}{(\eps^2 t)^{\frac{1+\beta}{2}}},
    \end{equation*}
    so the last part of~\eqref{eq:heat_kernel_differences} holds in either case.
\end{proof}

\section{Numerical methods} \label{app:numerical_methods}

The figures in this paper were generated by numerically solving the viscous Burgers equation
\begin{equation*}
    \partial_t m^\eps + \partial_x\biggl(\frac{1}{2}(m^\eps)^2\biggr) = \frac{\eps^2}{2} \partial_{xx} m^\eps
\end{equation*}
and the associated SDEs
\begin{equation*}
    dX_t^\eps = \frac{1}{2} m^\eps(X_t^\eps, t)\, dt + \eps\, dW_t, \qquad dY_t^\eps = m^\eps(Y_t^\eps, t)\, dt + \eps\, dW_t.
\end{equation*}

The viscous Burgers equation was solved on a finite computational domain $[-L, L]$ with Dirichlet boundary conditions using a standard finite difference method. The spatial derivatives were discretized using a first-order upwind scheme for the advection term and a second-order central difference for the diffusion term. The forward Euler scheme was used for the temporal derivative, where the time step was chosen dynamically at the beginning of the simulation to satisfy the Courant-Friedrichs-Lewy (CFL) stability condition for both the advection and diffusion terms. The computational domain was chosen to be sufficiently large such that boundary effects did not influence the evolution of the shock for the duration of the simulation.

The SDEs were solved for an ensemble of paths using the Euler--Maruyama method. The drift terms were evaluated at off-grid points using bilinear interpolation from the discrete solution of the PDE. A common Brownian motion was used for both SDEs and all paths.

The pushforward density $u^\eps = X^\eps_{\#} u_\initial$ was computed using a discretization of the formula $u^\eps(X_t) = u_\initial(X_0) |\partial X_0 / \partial X_t^\eps|$, i.e.~multiplying the initial density by a compression factor approximated by comparing the initial spacing of the numerical particles to their final spacing.

\begin{table}[h]
    \centering
    \begin{tabular}{l l l}
        \multicolumn{3}{l}{\textbf{PDE solver parameters}} \\
        \hline
        $L$ & Domain half-length & $6.0$ \\
        $T$ & Final time & $1.0$ \\
        $N_x$ & Number of spatial grid points & $1201$ \\
        \hline
        \multicolumn{3}{l}{\textbf{SDE solver parameters}} \\
        \hline
        $N_p$ & Number of stochastic paths & $4001$ \\
        $N_t$ & Number of time steps & $4000$ \\
        \hline
        \multicolumn{3}{l}{\textbf{Monte Carlo simulation}} \\
        \hline
        $N_{MC}$ & Number of Monte Carlo runs & $5000$ \\
        \hline
    \end{tabular}
    \caption{Default numerical parameters used for the simulations.} \label{table:default_params}
\end{table}

\printbibliography
\end{document}